\documentclass[10pt]{article}
\usepackage{mathrsfs}
\usepackage{amsfonts}
\usepackage{titlesec}
\usepackage{enumerate}
\usepackage[toc,page,title,titletoc,header]{appendix}
\usepackage{amsthm,amsmath,amssymb,anysize}
\usepackage[numbers,sort&compress]{natbib}
\theoremstyle{definition}
\newtheorem{theorem}{Theorem}
\newtheorem{lemma}{Lemma}[section]
\newtheorem{remark}[lemma]{Remark}
\newtheorem{proposition}[lemma]{Proposition}
\newtheorem{definition}[lemma]{Definition}
\newtheorem{corollary}[lemma]{Corollary}
\newcommand{\periodafter}[1]{#1.}
\titleformat{\subsection}[runin]
{\normalfont\bfseries}{\thesubsection}{0.5em}{\periodafter}
\marginsize{4.2cm}{4.2cm}{3.6cm}{3cm}
\numberwithin{equation}{section}
\renewcommand{\thelemma}{\arabic{section}.\arabic{lemma}}
\renewcommand{\thesection}{\arabic{section}.}
\renewcommand{\thesubsection}{\arabic{section}.\arabic{subsection}.}

\newcounter{RomanNumber}

\newcommand{\MyRoman}[1]{\setcounter{RomanNumber}{#1}\Roman{RomanNumber}}
\title{\textsf{Maximal Subalgebras for Modular  Graded Lie Superalgebras of Odd Cartan Type }}
\author{\textsc{Wende Liu$^{1,2}$\footnote{Correspondence: wendeliu@ustc.edu.cn}\setcounter{footnote}{-1}\footnote{Supported by the NSF of China (11171055) and the NSF of the Education Department of HLJ Province (12521158).} and Qi Wang$^1$}\\
{\small \textit{$^1$Department of Mathematics, Harbin Institute of Technology,}}\\
\small \textit{Harbin 150006, China} \\
 \small\textit{$^2$School of Mathematical Sciences, Harbin Normal University}\\
 \small\textit{Harbin 150025, China}}
\date{ }

\begin{document}
\maketitle
\begin{quotation}
\small\noindent \textbf{Abstract}: The purpose of this paper is to determine
all the maximal graded subalgebras of the
four infinite families of finite-dimensional  graded Lie superalgebras
of odd Cartan type over an algebraically closed field of prime characteristic $p>3$.
All the maximal graded subalgebras consist of
three types (\MyRoman{1}),  (\MyRoman{2})  and  (\MyRoman{3}).
In addition,  the maximal graded subalgebras of type (\MyRoman{3})
consist of the maximal graded R-subalgebras and the  maximal graded S-subalgebras.
In this paper we classify  all the maximal graded subalgebras for modular  Lie superalgebras of odd Cartan type.
We also determine the number of conjugacy classes, representatives of conjugacy classes
and the dimensions for the maximal graded R-subalgebras and the maximal graded subalgebras of type  (\MyRoman{2}).
Here, determining the maximal graded S-subalgebras is reduced
to determining the maximal subalgebras of classical Lie superalgebras $\mathfrak{p}(n)$.

\vspace{0.2cm} \noindent{\textbf{Keywords}}: Lie superalgebra of
odd Cartan type; Maximal subalgebra

\vspace{0.2cm} \noindent{\textbf{Mathematics Subject Classification 2010}}: 17B50, 17B05, 17B66, 17B70
\end{quotation}

\setcounter{section}{-1}
\section{Introduction}
The maximal subsystems of an algebraic or geometric system can be used perfectly
to characterize the system.
For example,
the classification problem of primitive transformation groups,
posed by S. Lie in the 19th century,
can be reduced to the determination of certain maximal subgroups in Lie groups.
The classification problem of the  maximal subalgebras (resp. subgroups)
for a given algebra  (resp. group) is therefore a typical one.

This problem has been attracted a great attention and produced some beautiful results.
The early work was the classification of maximal subalgebras (resp. subgroups)
for complex semisimple Lie algebra (resp. classical groups),
which has been developed in a different way by many authors
(cf., for example, E.B. Dynkin \cite{d11,d12}, F.I. Karpelevich \cite{fi}, A.I. Maltsev \cite{ai} and  V.V. Morozov \cite{vv}).
In particular, the papers of Dynkin \cite{d11,d12} are representative,
where the maximal subgroups of some classical groups
and the maximal subalgebras of complex semisimple Lie algebras are classified.
The maximal subalgebras
of  finite-dimensional  central simple algebras were studied by A. Elduque \cite{ae} and M. Racine \cite{mra}.
The classification for maximal subalgebras of associative superalgebras,  Jordan superalgebras and
Lie superalgebras has been developed by many authors
(cf., for example, A. Elduque, J. Laliena and S. Sacristan \cite{alb, alb1}, I. Shchepochkina \cite{is1,is2}).
In particular,
A. Elduque, J. Laliena and S. Sacristan \cite{alb,alb1} determined the maximal subalgebras
of finite-dimensional  central simple superalgebras.
I. Shchepochkina \cite{is1} determined the maximal subalgebras of  linear Lie superalgebras.

However, the classification of
 the maximal subalgebras (resp. subgroups) of  simple Lie (super)algebras over a positive characteristic field
 (resp. algebraic groups) is difficult to solve.
Because of  the great importance of the maximal subalgebras in  the classification theory of modular Lie (super)algebras,
there are several attempts to classify
 the maximal subalgebras of simple modular Lie (super)algebras.
The classification problem of the  maximal subalgebras
for Lie $p$-algebras
and the maximal subgroups for simple algebraic groups over positive characteristic field has been studied by many authors
(cf., for example, M. I. Kuzenesov \cite{mi}, H. Melikyan \cite{HM},
G. M. Seitz \cite{gm1,gm2},  O. Ten \cite{ot} and B. Weisfeiler \cite{bw}).
In particular, Melikyan \cite{HM} classified
the maximal graded subalgebras for modular Lie algebras of Cartan type
$W$, $S$, $H$, $K$ over an algebraically closed field.
Furthermore,   Melikyan's work was generalized to the situation of modular Lie superalgebras of  Cartan type
$W$, $S$, $H$ and $K$  \cite{bai}.

This paper is devoted to classifying
all the maximal graded subalgebras of Lie superalgebras of odd Cartan type $HO$, $SHO$, $KO$ and $SKO$ (for a definition, see Section 1).
Unlike in the case of modular Lie superalgebras of Cartan type $W$, $S$, $H$ and $K$,
the modular Lie superalgebras of odd Cartan type are defined by odd differential forms,
which are analogous  to the infinite-dimensional simple linearly compact Lie superalgebras
$HO$, $SHO$, $KO$ and $SKO$ (see \cite{k2}) over a field of characteristic 0.
Generally speaking, the structures of modular Lie superalgebras of odd Cartan type are  more complicated.
In contrast to other modular Lie superalgebras of Cartan type,
for such a Lie superalgebra, the Lie bracket induces an odd skew supersymmetric, non-degenerate bilinear form on
the $-1$th component of the standard gradation, which is crucial to determine the isomorphic classes of maximal graded subalgebras.
One may easily find a rough classification of the maximal graded subalgebras for Lie superalgebras of odd Cartan type.
Let $L$ be such a Lie superalgebra of the standard gradation
$$
L=L_{-2}+L_{-1}+\cdots
$$
with convention that $L_{-2}=0$ for $L=HO$ or $SHO$.
A nontrivial subalgebra $M$ is called a
\textit{maximal graded subalgebra} ({MGS}) provided that $M$ is
$\mathbb{Z}$-graded and that no nontrivial $\mathbb{Z}$-graded subalgebras  of $L$ contains strictly $M$.
Note that
$L_{-1} $ is an irreducible $L_{0}$-module.
It follows that
$$
\oplus_{i\geqslant0}L_{i} ~\mbox{is an MGS of $L$}.
$$
Clearly, any other MGS, $M$, must fulfill exactly one of the
following conditions:
\begin{eqnarray*}
&\mathrm{(\MyRoman{1})}& M_{-1}=L_{-1} ~\mbox{and}~ M_{0}=L_{0};\\
&\mathrm{(\MyRoman{2})}& M_{-1}~\mbox{is a nontrivial subspace of $L_{-1}$};\\
&\mathrm{(\MyRoman{3})}& M_{-1}=L_{-1}  ~\mbox{but}~ M_{0}\neq L_{0}.
\end{eqnarray*}

The paper is arranged as follows: The first section is to set some basic notions.
In Section 2, we establish a $\mathbb{Z}$-homogeneous automorphism
by the odd skew supersymmetric, non-degenerate bilinear form on
the $-1$th component of the standard gradation and study the weight space decompositions.
In Section 3, we  construct  and determine  all the MGS of type (I).
In Section 4, we first establish  a sufficient and necessary condition for
a graded subalgebra of type (\MyRoman{2}) to be maximal
and then give the number of conjugacy classes, representatives of conjugacy classes
and the dimensions for the maximal subalgebras of type (\MyRoman{2}).
In Section 5, we still first give a sufficient and necessary condition
for a graded subalgebra to be maximal R-subalgebras of type (\MyRoman{3})
and then give  the number of conjugacy classes, representatives of conjugacy classes
and the dimensions of these maximal subalgebras.
Finally,  the classification of the  maximal graded S-subalgebras of type (\MyRoman{3}) for $L$ is reduced to the classification of the maximal subalgebras of classical Lie superalgebras $\mathfrak{p}(n)$.

We close this introduction by introducing the general conventions.
The underlying field $\mathbb{F}$ is an algebraically closed field of prime characteristic $p>3$.
In addition to the standard notation $\mathbb{Z}$,
we write $\mathbb{N}$ and $\mathbb{N}_0$ for
the sets of nonnegative integers, positive integers, respectively.
The field of two elements is denoted by $\mathbb{Z}_{2}=\{\overline{0}, \overline{1}\}.$
For any $m,n\in \mathbb{Z}$, write $\overline{m,n}=\{m,m+1,\ldots, n\}$, if $m\leq n$ and $\overline{m,n}=\emptyset$, otherwise.
Write $\langle X \rangle$ for the ($\mathbb{Z}_{2}$-graded) subalgebra generated by a subset $X$ of a superalgebra.
All vector spaces are finite-dimensional. All the subspaces, subalgebras and submodules are assumed to be $\mathbb{Z}_{2}$-graded
and all the homomorphisms of superalgebras
are both $\mathbb{Z}_{2}$-homogeneous and $\mathbb{Z}$-homogeneous.
For a $\mathbb{Z}_2$-graded vector space $V$ and $x\in V$, the symbol $|x|$ implies that
$x$ is a $\mathbb{Z}_2$-homogeneous element and $|x|$ is the parity.

\section{Lie superalgebras of odd Cartan type}\label{sec1}
In this section we introduce simple   Lie superalgebras of odd Cartan type.
Fix two positive integers $m$ and $n$.
Put
$$
\mathbf{I}_0=\{ 1,2,\ldots,m \},\quad \mathbf{I}_1
 =\{m+1,\ldots,m+n\}\quad \mbox{and}\quad \mathbf{I}=\mathbf{I}_0\cup \mathbf{I}_1.
$$
Write
$$
\mathbf{A}(m)=\{\alpha=(\alpha_{1},\ldots,\alpha_{m})\in\mathbb{N}^{m}\mid
0\leqslant\alpha_{i}\leqslant p-1,i\in\mathbf{I}_{0} \}.
$$
Let $\mathcal {O}(m)$ be the \textit{divided power algebra} with $\mathbb{F}$-basis
$
\{x^{(\alpha)}\mid\alpha\in\mathbf{A}(m)\}
$.
Let $\Lambda(n)$ be the \textit{exterior superalgebra} over $\mathbb{F}$ of $n$ variables
$x_{m+1},x_{m+2},\ldots,x_{m+n}.$
Then the tensor product
$$
\mathcal {O}(m,n)=\mathcal {O}(m)\otimes\Lambda(n)
$$
is an associative superalgebra with respect to the usual $\mathbb{Z}_{2}$-gradation.

Let
$$
\mathbf{B}(n)=\{\langle i_{1},i_{2},\ldots,i_{k}\rangle\mid 0\leqslant  k\leqslant  n;
m+1\leqslant i_{1}<i_{2}<\cdots<i_{k}\leqslant m+n\}
$$
be the set of $k$-shuffles of strictly increasing integers in $\mathbf{I}_1$, $0\leqslant k\leqslant n.$

For $u=\langle i_1,i_2,\ldots,i_k\rangle\in \mathbf{B}(n),$
write $|u|=k$ and $x^{u}=x_{i_{1}}x_{i_{2}}\cdots x_{i_{k}}$ ($x^{\emptyset}=1$).
If $g\in\mathcal {O}(m)$ and $f\in \Lambda(n)$
we simply write $gf$ for $g\otimes f$.
Then $\mathcal {O}(m,n)$ has a $\mathbb{Z}_{2}$-homogeneous $\mathbb{F}$-basis
$$
\{x^{(\alpha)}x^{u}\mid \alpha\in \mathbf{A}(m),u\in
\mathbf{B}(n)\},
$$
which is called the \textit{standard basis} of $\mathcal {O}(m,n)$.

For $i\in \textbf{I}_{0}$ and
$\varepsilon_{i}=(\delta_{i1},\delta_{i2},\ldots,\delta_{im})$,
write $x_{i}$ for
$x^{(\varepsilon_{i})}.$
Let $\partial_{1},\ldots,\partial_{m+n}$
be the \textit{superderivations} of the superalgebra
$\mathcal {O}(m,n)$ such that $\partial_{i}(x_j)=\delta_{ij}.$
The \textit{parity} of $\partial_{i}$ is $|\partial_{i}|=\tau(i)$,
where
$\tau \left( i\right)$ is $\bar{0}$ if $i\in \mathbf{I}_{0}$ and $\bar{1}$ if $i\in \mathbf{I}_{1}$.
Put
$$
W(m,n)=\Big\{ \sum_{i\in \mathbf{I}}a_i \partial_i \mid a_i\in \mathcal{O} ( m,n), i\in \mathbf{I}\Big\}.
$$
Then $W(m,n)=\mathrm{Der}_{\mathbb{F}}\mathcal{O}(m,n)$
is a finite-dimensional  simple Lie superalgebra, called the \textit{generalized Witt superalgebra}.
We simply write $\mathcal{O}$ and $W$ for $\mathcal{O}(m,n)$ and $W(m,n)$, respectively.
The so-called Lie superalgebras of odd Cartan type  are reviewed as follows.

\subsection{Odd Hamiltonian Lie superalgebras}
We introduce the odd Hamiltonian Lie superalgebras.
Suppose $n\geq 3$ and
define a linear operator
$D_{HO}:\mathcal{O}(n,n)\longrightarrow W(n,n)$ such that
$$
D_{HO}(f)=\sum_{i=1}^{2n}(-1)^{\tau(i)|f|}\partial_i(f)\partial_{i'},
$$
where
$'$ is an involution of the index set $\{1,\ldots,2n\}$
sending $i$ to $i+n$ for $i\leq n$.
Obviously, $D_{HO}$ is odd with kernel $\mathbb{F}\cdot 1$, and
$$
[D_{HO}(f), D_{HO}(g)]=D_{HO}(D_{HO}(f)(g)), ~\mbox{for all}~   f, \ g\in \mathcal{O}(n,n).
$$
Write $\bar {\mathcal{O}}(n,n)$ for the quotient superspace $\mathcal{O}(n,n)/\mathbb{F}\cdot 1$. Then
$D_{HO}$ induce a linear operator of $\bar {\mathcal{O}}(n,n)$ into $W(n,n)$, which is denoted still by $D_{HO}$. Note that $HO(n)=\bar {\mathcal O}(n,n)$ is a
finite-dimensional  simple Lie superalgebra with respect to bracket:
$$
[f,g]=D_{HO}(f)(g), ~\mbox{for all}~ f,\  g \in HO(n).
$$
$HO(n)$, or simply $HO$, is called the \textit{odd Hamiltonian Lie superalgebra} \cite{Liu1}.
It is easy to see  that $HO(n)$ has a $\mathbb{Z}$-gradation structure
induced by $\mathrm{zd}(x_i)=-1$, since $\mathrm{zd}(D_{HO})=-2$.

We assemble some conclusions in   \cite[Proposition 1 and Theorem 5]{Liu1} to be the following lemma.
\begin{lemma}\label{1.1}
 The following statements hold:
\begin{enumerate}
\item[(1)]$\dim HO(n)=2^np^n-1$.
\item[(2)]$HO(n)$ is transitive.
\item[(3)]$HO(n)=\langle L_{-1}+L_0+L_1\rangle$.
\item[(4)]$T_{HO}=\mathrm{span}_{\mathbb{F}}\{x_ix_{i'}\mid i\in \overline{1,n}\}$ is a torus of $HO(n)$. \qed
\end{enumerate}
\end{lemma}
\subsection{Special odd Hamiltonian Lie superalgebras}
We introduce the special odd Hamiltonian Lie superalgebras.
Suppose $n\geq 3$ and
let $\mathrm{div}: W(n,n)\longrightarrow\mathcal{O}(n,n)$ be the linear operator, called the \textit{divergence}, such that
$$\mathrm{div}(f_r\partial_r)=(-1)^{|\partial_r||f_r|}\partial_r(f_r).$$
Note that $\mathrm{div}$ is an even superderivation of $W(n,n)$ into module $\mathcal{O}(n,n)$.
The kernel of $\mathrm{div}$ is a $\mathbb{Z}$-graded subalgebra of $W(n,n)$, whose
derived algebra is a simple Lie superalgebra,
which is called the \textit{special superalgebra}.
Recall that the odd Laplacian operator $\Delta=\sum_{i=1}^{n}\partial_{i}\partial_{i'}.$
Then
\begin{eqnarray*}
 SHO'(n)=\{a\in \bar{\mathcal{O}}(n,n)\mid\mathrm{div}\; \mathrm{D_{HO}}(a)=0\}=\{a\in \bar{\mathcal{O}}(n,n)\mid\Delta(a)=0\}
\end{eqnarray*}
is a sub Lie superalgebra of $HO(n)$, whose derived algebra is denoted by $\overline{SHO}(n)$.
The two-step derived algebra of $SHO'(n)$, denoted by $SHO(n)$, or simply $SHO$, is called the
\textit{special odd Hamiltonian  Lie superalgebra} \cite{lh}.
$SHO$ has a $\mathbb{Z}$-gradation structure induced by
$\mathrm{zd}(x_i)=-1$.
We assemble some conclusions in \cite[Theorems 4.1 and 4.7]{lh} to be the following lemma.
\begin{lemma}\label{1.2}
Let $L=SHO(n)$, then the following statements hold:
\begin{enumerate}
\item[(1)] The following formulas hold:
\begin{align}
\dim L'&=\sum^{n}_{l=2}\bigg((2^{n-1}-2^{n-l})\sum_{(i_1,\ldots,i_l)\in\mathbf{J}(l)}\prod_{k=1}^l\pi_{i_k}\bigg)+(p+1)^n-1;\nonumber\\
\dim  \bar L&=\sum^{n}_{l=2}\bigg((2^{n-1}-2^{n-l})\sum_{(i_1,\ldots,i_l)\in\mathbf{J}(l)}\prod_{k=1}^l\pi_{i_k}\bigg)+(p+1)^n-2^n-1;\nonumber\\
\dim   L&=\sum^{n}_{l=2}\bigg((2^{n-1}-2^{n-l})\sum_{(i_1,\ldots,i_l)\in\mathbf{J}(l)}\prod_{k=1}^l\pi_{i_k}\bigg)+(p+1)^n-2^n-2,\nonumber
\end{align}
where $\mathbf{J}(l)=\{(i_1,\ldots,i_l)\mid 1\leq i_1<\cdots< i_l\leq n\}$ and $\pi_i=p-1$, $i\in\overline{1,n}$.
\item[(2)]$L$ is transitive.
\item[(3)]$L=\langle L_{-1}+L_0+L_1\rangle$.
\item[(4)]$T_{SHO}=\mathrm{span}_{\mathbb{F}}\{x_ix_{i'}-x_{i+1}x_{(i+1)'}\mid i\in \overline{1,n-1}\}$ is a torus of $L$. \qed
\end{enumerate}
\end{lemma}
\subsection{Odd Contact Lie superalgebras}
We introduce the odd Contact Lie superalgebras.
Suppose $n\geq3$, and define a linear operator
$D_{KO}:\mathcal{O}(n,n+1)\longrightarrow W(n,n+1)$ such that
$$
D_{KO}(f)=D_{HO}(f) + (-1)^{|f|}\partial_{2n+1}(f)\frak{D} +(\frak{D}(f)-2f) \partial_{2n+1},
$$
where $\frak{D}=\sum_{i=1}^{2n}x_i\partial_i$ and
  $D_{HO}(f)=\sum_{i=1}^{2n}(-1)^{\tau(i)|f|}\partial_i(f)\partial_{i'}$, $f\in \mathcal{O}(n,n+1)$.
Note that for $f\in \mathcal{O}(n,n+1)$ we write $|f|$ for the $\mathbb{Z}_2$-parity of $f$ in $\mathcal{O}(n,n+1)$.
Obviously, $D_{KO}$ is odd.
One sees that
$KO(n)=\mathcal{O}(n,n+1)$, or simply $KO$, is a
finite-dimensional  simple Lie superalgebra, with the bracket:
$$
[f,g]=D_{KO}(f)(g) - (-1)^{|f|}2 \partial_{2n+1}(f)(g),  ~\mbox{for all}~ f,\  g \in KO(n),
$$
called the \textit{odd Contact   Lie superalgebra} \cite{FJZ}.
$KO(n)$ has a $\mathbb Z$-gradation induced by $\mathrm{zd}(1)=-2$,
$\mathrm{zd}(x_i)=-1+\delta_{i,2n+1}$.
As for $HO(n)$, we write down the following lemma assembled from \cite[Theorem 3.1 and Corrollary 3.1]{FJZ}.
\begin{lemma}\label{th:1.2}
The following statements hold:
\begin{enumerate}
\item[(1)]$\dim KO(n)=2^{n+1}p^n$.
\item[(2)]$KO(n)$ is transitive.
\item[(3)]$KO(n)=\langle L_{-1}+L_0+L_1\rangle$.
\item[(4)]$T_{KO}=\mathrm{span}_{\mathbb{F}}\{x_ix_{i'}, x_{2n+1}\mid i\in \overline{1,n}\}$ is a  torus of $KO(n)$.  \qed
\end{enumerate}
\end{lemma}
\subsection{Special odd Contact Lie superalgebras}
We introduce the special odd Contact Lie superalgebras.
Suppose $n\geq3$, and for any $\lambda\in \mathbb{F}$, let
$$
\mathrm{div}_{\lambda}(u)=(-1)^{|u|}2\left(\Delta
\left(u\right)+\left(\frak{D}-n\lambda\mathrm{id}_{\mathcal{O}}\right)
\partial_{2n+1}\left(u\right)\right), \quad u\in \mathcal{O}(n,n+1).
$$
The kernel of $\mathrm{div}_{\lambda}$ is a sub Lie superalgebra of $KO(n)$,
denoted by $SKO''(n,\lambda)$,  whose two-step derived algebra is
simple, called  the \textit{special odd Contact   Lie superalgebra} \cite{ly1},
denoted by $SKO(n, \lambda)$ or simply, $SKO$. Note that $SKO(n,\lambda)$ has a
$\mathbb{Z}$-grading structure induced by $\mathrm{zd}(1)=-2$ and
$\mathrm{zd}(x_i)=-1+\delta_{i,2n+1}$.
Put $\pi_i=p-1$, for $i\in\overline{1,n}$.
We assemble some conclusions in \cite[Theorems 3.1 and 3.2]{ly1} to be the following lemma:
\begin{lemma}
Let $L=SKO(n,\lambda)$. Then the following statements hold:
\begin{enumerate}
\item[(1)]
\begin{align}
\dim L&=2\left(\sum_{l=2}^n\bigg(\big(2^{n-1}-2^{l-1}\big)\sum_{(i_1,i_2\ldots, i_l)\in \mathbf{J}(l)}\prod_{c=1}^l\pi_{i_c}\bigg)+\prod_{j=1}^n(\pi_j+2)\right)\nonumber\\
    &-\sum_{k_i\in\frak{G}_2(\lambda,n)}\binom{n}{k_i}-2^n-\delta'_{n\lambda,-1}\nonumber
\end{align}
where
\begin{align}
&\mathbf{J}(l)=\{(i_1,\ldots,i_l)\mid 1\leq i_1<\cdots<i_l\leq n)\}, \mathbf{J}(0)=\emptyset,\nonumber\\
&\frak{G}_2(\lambda,n)=\{k\in \overline{0,n}\mid n\lambda-n+2k+2=0\in \mathbb{F}\}\nonumber
\end{align}
and
$$\delta'_{n\lambda, -1}=
\begin{cases}
1, &  n\lambda+1=0 ~\mbox{in}~ \mathbb{F},\\
0, &  n\lambda+1\neq0 ~\mbox{in}~ \mathbb{F}.
\end{cases}$$

\item[(2)]$L$ is transitive.
\item[(3)]$L=\langle L_{-1}+L_0+L_1\rangle$.
\item[(4)]$T_{SKO}=\mathrm{span}_{\mathbb{F}}\{x_ix_{i'}-x_{i+1}x_{(i+1)'},  x_{2n+1}+n\lambda x_1x_{1'}\mid i\in \overline{1,n-1}\}$ is a  torus of $L$.  \qed
\end{enumerate}
\end{lemma}

 The finite-dimensional $\mathbb{Z}$-graded Lie superalgebras
$HO(n)$, $SHO(n)$, $KO(n)$ and $SKO(n,\lambda)$  are called the \textit{Lie superalgebras of odd Cartan type.}
In this paper, we suppose $n>3$ for simplicity, although sometimes a weaker hypothesis is sufficient.
We call $T_{L}$  the \textit{standard torus of} $L$, where $L$ is a Lie superalgebra of odd Cartan type.
\section {Automorphisms and weight spaces}
\subsection{Automorphisms}
Let $L$ be a   Lie superalgebra of odd Cartan type.
We define a bilinear form
$(-,-): L_{-1}\times L_{-1}\longrightarrow \mathbb{F}$
by
$$(x_i+\mathbb{F}, x_j+\mathbb{F})\longmapsto D_{HO}(x_i)(x_j)$$
 for $L=HO$ or $SHO$, and by
$$(x_i, x_j)\longmapsto [x_i,x_j]$$ for $L=KO$ or $SKO$.
It is easy to see that $(-,-)$ is
an odd skew supersymmetric, nondegenerate bilinear form.
That is, $(L_{-1}, (-,-))$ is a symplectic superspace
(see Definition \ref{A.1} in Appendix).
Suppose $V$ is a symplectic superspace.
We write $\mathrm{R}(V)$ for the rank of the symplectic superspace $V$,
since the dimensions of maximal isotropic subspaces of a symplectic superspace are unique, from now on, we write $\mathrm{misdim}V$ for the dimension of the maximal isotropic subspace of $V$.
\begin{lemma}\label{th:1.3}
Let $L$ be a Lie superalgebra of odd Cartan type.
Suppose $V$ is a nontrivial subspace of $L_{-1}$.
Then there exists a $\mathbb Z$-homogenous automorphism
$\Phi \in \operatorname{Aut}(L)$
such that
$$
\left\{ \Phi(x_i)\mid i\in \overline{1,k} \cup \overline{1',k'} \cup
 \overline{k+1,k+t}\cup \overline{(k+t+1)',s'}\right\}
$$
is a homogenous basis of $V$,
where $k=\mathrm{R}(V)/2$,
$t=\dim V_{\bar 0}-k$
and $s=\mathrm{misdim}V$.
\end{lemma}
\begin{proof}
Firstly, suppose $L=HO(n)$.
According to  results in linear algebra (see Lemma  \ref{th:0.19}),
there exists a homogenous basis of $V$
$$
\{y_1, \ldots, y_k,y_{k+1}, \ldots, y_{k+t}\mid y_{1'}, \ldots, y_{k'}, y_{(k+t+1)'}, \ldots, y_{s'}\}
$$
such that, for $i\in\overline{1,k+t}$, $j\in \overline{1',k'}\cup\overline{(k+t+1)',s'}$,
$$
(y_i, y_{j})=
\begin{cases}
1, &  i \in \overline{1,k},\quad j=i',\\
0, & \text{otherwise}.
\end{cases}
$$
Set
$$
V_1=\mathrm{span}_{\mathbb{F}}\{y_1, \ldots, y_k \mid y_{1'}, \ldots, y_{k'}\}.
$$
Then $L_{-1} =V_1\oplus V_1^{\bot}$.
Set
$$
V_2=\mathrm{span}_{\mathbb{F}}\{y_{k+1}, \ldots, y_{k+t}\mid y_{(k+t+1)'}, \ldots, y_{s'}\}.
$$
We have $V_2\subset V_1^{\bot}$ and $V=V_1\oplus V_2$.
Any homogeneous basis of $V_2$ can be extended to
a homogeneous symplectic basis of $V_1^{\bot}$ (see Lemma  \ref{th:0.4}).
Thus, the homogeneous basis of $V$ can be extended to a homogeneous symplectic basis of $HO_{-1}$: $\{y_1, \ldots, y_n\,|\,y_{1'}, \ldots, y_{n'}\}$,
where $|y_{i}|=|x_i|$.
So there must be an $2n\times 2n$ even matrix $A=(a_{ij})_{2n \times 2n}$ such that
$$
(x_1, \ldots, x_n, x_{1'}, \ldots, x_{n'})=(y_1, \ldots, y_n, y_{1'}, \ldots, y_{n'})A.
$$
Since $A$ is symplectic, for $ k\in \overline{1,n}$,
$l\in \overline{1',n'}$,
we have (see Lemma \ref{0.8})
$$
\sum_{i=1}^{n}a_{ik}a_{i'l} = \begin{cases}
1, &  l=k', \\
0, & \text{otherwise}.
\end{cases}
$$
Consider the mapping $\Phi: \mathcal{O}\longrightarrow \mathcal{O}$
such that $x_i\longmapsto y_i$,
for $i\in \overline{1,2n}$.
Since $|x_i|=|y_i|$,
$\Phi$ can be extended to a $\mathbb{Z}$-homogeneous automorphism of $\mathcal{O}$ (see \cite[Lemma 2.5]{LZA}).
Furthermore, this automorphism can induce a $\mathbb{Z}$-homogeneous  automorphism (as vector space) of $HO$,
which is still denoted as $\Phi$.
Assert that $\Phi D_{HO}(f)\Phi^{-1}=D_{HO}(\Phi(f))$, for all $f \in \mathcal O$.
In fact,
put
$$
(E_1,\ldots, E_n, E_{1'}, \ldots, E_{n'}) = (\partial_1, \ldots, \partial_n, \partial_{1'}, \ldots, \partial_{n'})A^T.
$$
For all $f \in L$, we have
\begin{eqnarray}
\Phi\big(D_{HO}(f)\big)\Phi^{-1} & = & \Phi \left (\sum_{i=1}^{2n}(-1)^{\tau(i')|f|}\partial_{i'}(f)\partial_i\right)\Phi^{-1}\nonumber\\ &=& \sum_{i,j=1}^{2n}(-1)^{\tau(i')|f|}a_{ij}\Phi\big(\partial_{i'}(f)\big)\partial_j \nonumber \\
& = & \sum_{i=1}^{2n}(-1)^{\tau(i')|f|}\Phi\big(\partial_{i'}(f)\big) \Bigg(\sum_{j=1}^{2n}a_{ij}\partial_j\Bigg) \nonumber \\
& =& \sum_{i=1}^{2n} (-1)^{\tau(i')|f|}E_{i'}\big(\Phi(f)\big)E_i \nonumber \\
& = & \sum_{i=1}^{2n}(-1)^{\tau(i')|f|} \Bigg( \bigg(\sum_{k=1}^{2n} a_{i'k} \partial_k \big(\Phi(f)\big)\bigg)\bigg(\sum_{l=1}^{2n} a_{il} \partial_l\bigg) \Bigg) \nonumber \\
& = & \sum_{i=1}^{2n}(-1)^{\tau(i')|f|}\Bigg(\sum_{k=1}^{2n} \sum_{l=1}^{2n} a_{i'k}a_{il} \Big(\partial_k\big(\Phi(f)\big)\partial_l\Big)\Bigg) \nonumber \\
& = & \sum_{k=1}^{2n} \sum_{l=1}^{2n}\left(\bigg(\sum_{i=1}^{2n}(-1)^{\tau(i')|f|} a_{i'k}a_{il}\bigg) \partial_k\big(\Phi(f)\big)\partial_l \right)\nonumber\\
&=&\sum_{k=1}^{2n}(-1)^{\tau(k)|f|}\partial_k\big(\Phi(f)\big)\partial_{k'} \nonumber\\
&=&D_{HO}\big(\Phi(f)\big). \nonumber
\end{eqnarray}
So $\Phi([f,g]) =[\Phi(f),\Phi(g)]$, for all $f,g\in L$.
Thus $\Phi$ is a desired automorphism.

Then we consider $L=SHO$. A direct calculation shows that
$$\mathrm{div}D_{HO}(\Phi(f))=\Phi\mathrm{div}D_{HO}(f)\Phi^{-1}, ~\mbox{for all}~ f\in HO.$$
So, $\Phi(SHO')\subset SHO'$.
Thus $\Phi|_{SHO}$ is a desired automorphism of $SHO$.

Next we suppose $L=KO$.
Let
$$
\{y_1, \ldots, y_k,y_{k+1}, \ldots, y_{k+t}\mid y_{1'}, \ldots, y_{k'}, y_{(k+t+1)'}, \ldots, y_{s'}\}
$$
be a homogeneous symplectic basis of $V$,
which can be extended to a homogeneous symplectic basis of $L_{-1}$,
denoted by
$\{y_1, \ldots, y_n \mid y_{1'}, \ldots, y_{n'}\}$.
So there exists a  $2n\times 2n$ even matrix $A=(a_{ij})_{2n \times 2n}$ such that
$$
 (y_1, \ldots, y_n, y_{1'}, \ldots, y_{n'}) = (x_1, \ldots, x_n, x_{1'}, \ldots, x_{n'})A^{-1}.
$$
For all $i\in \overline{1,2n}$,
consider the mapping $\Phi:x_i\longmapsto y_i$,
and $x_{2n+1}\longmapsto x_{2n+1}$.
Since $|x_{i}|=|y_{i}|$,
$\Phi$ can be extended to a $\mathbb{Z}$-homogeneous automorphism of $\mathcal{O}$,
which is still denoted as $\Phi$.
It is known from calculation that for all $ f\in L$,
the following equations hold:
\begin{align}
\Phi(D_{HO}(f))\Phi^{-1}&=D_{HO}(\Phi(f)), \nonumber \\
\Phi(\partial_{2n+1}(f)\frak{D})\Phi^{-1}&=\partial_{2n+1}(\Phi(f))\frak{D},\nonumber \\
\Phi((\frak{D}(f)-2f)\partial_{2n+1})\Phi^{-1}& =(\frak{D}(\Phi(f))-2\Phi(f))\partial_{2n+1}.\nonumber
\end{align}
Consequently, $\Phi$ is a desired automorphism of $L$.
It is easy to show that $\Phi|_{SKO}$ is a desired automorphism of $SKO$.
\end{proof}
\subsection{Weight space decompositions}
We discuss weight space decompositions of $L=HO(n)$.
Suppose $T$ is the standard torus of $L$. We choose a basis
$
\{\epsilon_i\mid i\in \overline{1,n}\}
$
of $T^*$,
where
$\epsilon_i(x_jx_{j'})=\delta_{j'i}-\delta_{ji}$ for $i,j\in \overline{1,n}$.
Put
$$\Psi=\mathbb{Z}\epsilon_1\oplus\mathbb{Z}\epsilon_2\oplus\cdots\oplus\mathbb{Z}\epsilon_n.$$
For simplicity, let $\epsilon_{i'}=-\epsilon_{i}$, for $i\in \overline{1,n}$.
The following lemma can be verified by a direct calculation:
\begin{lemma}\label{2.2}
Let
$L=L^\theta\oplus\bigoplus_{\gamma\in\Psi}L^\gamma$ and $L_i=L_i^\theta\oplus\bigoplus_{\gamma\in\Psi_i}L_{i}^{\gamma}$
be the weight space decompositions with respect to the torus $T$,
where $\Psi_i\subset\Psi\subset T^*$ and $\theta$ is the weight zero.
Then we have
\begin{enumerate}
\item[(1)] $\dim L_{-1}^{\epsilon_i}=1$, for $i\in \overline{1,2n}$.
\item[(2)] $\dim L_0^{\epsilon_i+\epsilon_j}=1$, for $i, j\in \overline{1,2n}$ and $i\neq j'$.
\item[(3)]
$
\dim L_1^{\epsilon_i+\epsilon_j+\epsilon_k}=\begin{cases}
1,  &~\mbox{for}~ i,j,k\in \overline{1,2n}, i\notin\{j',k'\} ~\mbox{and}~ j\neq k', \\
n, &\text{otherwise}.
\end{cases}
$
\end{enumerate}

\end{lemma}
The following lemma will be used frequently
which comes from the results in linear algebra (see also \cite{a}).
\begin{lemma}\label{th3.2}
The following statements hold:
\begin{enumerate}
\item[(1)]
Let $V$ be a vector space and
$\{f_1, \ldots, f_n\}$ be a finite set of linear functions on $V$.
Then
$$
\{x\in V\mid \prod_{1\leq i \neq j\leq n}(f_i-f_j)x\neq 0\}\neq \emptyset.
$$

\item[(2)] Let $L$ be Lie superalgebra.
For $a\in L$,
suppose $x=\sum_{i=1}^n x_i$ is a sum of eigenvectors of $\mathrm{ad}a$
associated with mutually distinct eigenvalues.
Then all $x_i$'s lie in the sub Lie superalgebra generated by $a$ and $x$.
\end{enumerate}
\end{lemma}

\section{MGS of Type $(\textrm{I})$}
Suppose $L$ is a Lie superalgebra
of odd Cartan type.
This section is concerned with MGS, $M$,
of type $(\textrm{I})$ for $L$,
that is,
$ M_{-1}=L_{-1} ~\mbox{and}~ M_{0}=L_{0}$.
Embed $KO(n)$ into $KO(n+1)$ in a natural fashion.
Without notice, we view $KO(n)$ as a subalgebra of $KO(n+1)$.
Put $\xi=np$. We have the following theorem:
\begin{theorem}\label{03.1}
Suppose $L$ is a Lie superalgebra
of odd Cartan type.  Then
\begin{enumerate}
\item[(1)] For $L=HO(n)$,
$L$ has exactly one MGS of type $(\mathrm{I}):$
$$
L_{-1}+L_{0}+SHO_1'+SHO_2'+\cdots+SHO_{\xi-2}'.
$$
\item[(2)] For $L=SHO(n)$,
$L$ has exactly one MGS  of type (\MyRoman{1}):
$$
L_{-1}+L_{0}.
$$
\item[(3)] For $L=KO(n)$, all the  MGS of type (\MyRoman{1}) are listed below:

$$
L_{-2}+L_{-1}+L_{0}+L_{1,0}+L_{2,0}+\cdots+L_{\xi-2,0}\nonumber \\
$$
and
\begin{equation*}
L_{-2}+L_{-1}+L_{0}+L^{\lambda}_{1}+\cdots+L^{\lambda}_{\xi-2},
\end{equation*}
where  $L_{i,0}=L_i\cap\mathcal{O}(n,n)$ and
\begin{align}
L_{i}^{\lambda}&=\begin{cases}
 SKO''\big(n, (\lambda+1)/n\big)_i, &    ~\mbox{if}~ n\not\equiv 0 (\mathrm{mod} \  p)\\
SKO''\big(n+1,(\lambda+1)/(n+1)\big)_i\cap KO(n),  &~\mbox{if}~ n\equiv 0 (\mathrm{mod} \  p)
\end{cases}\nonumber
\end{align}
for $\lambda\in \mathbb{F}$
and  $i>0$.
\item[(4)] For $L=SKO(n,\lambda)$,
\begin{enumerate}
\item if $1+n\lambda\neq 0$ in $\mathbb{F}$,
$L$ has exactly one MGS  of type (\MyRoman{1}):
$$
L_{-2}+L_{-1}+L_{0}+L_{1,0}+\cdots+L_{\xi-2,0},
$$
where
$$
L_{1,0}=L_1\cap\mathcal{O}(n,n).
$$
\item if $1+n\lambda=0$ in $\mathbb{F}$,
$L$ has exactly two MGS  of type (\MyRoman{1}):
$$
L_{-2}+L_{-1}+L_{0}+L_{1,0}+\cdots+L_{\xi-3,0},
$$
and
$$
L_{-2}+L_{-1}+L_{0}+L_{1,1},
$$
where
$$
L_{1,0}=L_1\cap\mathcal{O}(n,n), \quad
L_{1,1}=\{fx_{2n+1}\in L_1\mid f\in \mathcal{O}(n,n)\}.
$$
\end{enumerate}
\end{enumerate}
\end{theorem}

We will prove  this theorem in the following subsections.
\subsection{$HO(n)$}
 We will determine the MGS of type  (\MyRoman{1}) for $HO(n)$.
Obviously $SHO'_i$ is an $HO_0$-submodule, for $i\geq-1$.
Write
$U(HO_0)$ for the universal enveloping algebra of $HO_0$.
By a direct computation, we have the following lemma:：
\begin{lemma}\label{th:2.1}
The following statements hold:
\begin{enumerate}
\item[(1)]$SHO_1'$ is an irreducible $HO_0$-module.
\item[(2)]$HO_{1}=U(HO_{0})x^{(2\varepsilon_{i_{0}})}x_{i_{0}'}$, for $i_{0}\in\overline{1,n}.$
\item[(3)]$HO_{1}=U(HO_{0})D$, for $D\in HO_1\setminus SHO'_1$.
\end{enumerate}
\end{lemma}
\begin{remark}
For $n=3$, $SHO_1'$ is a reducible $HO_0$-module.\qed
\end{remark}

\begin{proof}[\textit{Proof of Theorem \ref{03.1}(1)}]
Let $M=HO_{-1}+HO_{0}+SHO_1'+SHO_2'+\cdots+SHO_{\xi-2}'.$
Obviously, $M$ is a proper subalgebra of $HO$.
Suppose $\overline{M}$ is an arbitrary subalgebra of $HO$ with $\overline{M}_{-1}=HO_{-1}$ and $\overline{M}_0=HO_0$.
If $\overline{M}\nsubseteq M$, then
there exists $D\in \overline{M}\setminus M$. Obviously, $\Delta(D)\neq 0$. Moreover,
there must be $k\in\overline{1,2n}$
such that
$$
\Delta[x_{k}, D]=-\partial_{k'}\Delta(D)\neq 0.
$$
Thus there is $D'\in \overline M_1\backslash SHO_1'$.
We can see that $\overline{M} \supset U(L_0)D'=L_1$ by Lemma \ref{th:2.1}.
Thus we have $\overline{M}=HO$, by Lemma \ref{1.1}.
Therefore, $M$ is the unique MGS of type (\MyRoman{1}).
\end{proof}
\subsection{$SHO(n)$}
 We will determine the MGS of type  (\MyRoman{1}) for $SHO(n)$.
Let $m\geq n$.
 For $i\in\overline{1,n}$,
define the linear operator $\nabla_i$ on $\mathcal{O}(n,m)$ such that
$$
\nabla_i(x^{(\alpha)}x^{u})=x^{(\alpha+\varepsilon_i)}x_{i'}x^{u},
$$
where
 $(\alpha, u)\in \mathbf{A}(n)\times\mathbf{B}(m)$.
 We also use the symbols
  $\Delta_i=\partial_i\partial_{i'}$ and $\Gamma_i^j=\nabla_j\Delta_i$,
for $i,j\in\overline{1,n}.$
Fix $(\alpha, u)\in \mathbf{A}(n)\times\mathbf{B}(m)$ and put
$$
\mathbf{I}(\alpha, u)=\{
i\in\overline{1,n}\mid \Delta_i(x^{(\alpha)}x^u)\neq0\},$$
$$
\mathbf{\widetilde{I}}
(\alpha, u)=\{
i\in\overline{1,n}\mid \nabla_i(x^{(\alpha)}x^u)\neq0\},
$$
$$
\mathcal{D^*}=\{x^{(\alpha)}x^{u}\in \mathcal{O}(n,m)\mid\mathbf{I}(\alpha, u)\neq
\emptyset,\mathbf{\widetilde{I}}
(\alpha, u)\neq \emptyset\}.
$$
Now suppose $m=n$ and put
$$
\mathcal{G}=\widetilde{\mathcal{G}}\cup\{x^{(\alpha)}x^{u}\in SHO\mid
\mathbf{I}(\alpha, u)=
\emptyset \},
$$
where
$$
\widetilde{\mathcal{G}}=\{x^{(\alpha)}x^{u}-\sum_{i\in\mathbf{I}(\alpha, u)}
\Gamma_i^q(x^{(\alpha)}x^{u})\in SHO \mid x^{(\alpha)}x^{u}\in \mathcal{D}^*,
q\in\widetilde{\mathbf{I}}(\alpha, u)\}.
$$
The following lemma follows from \cite[Theorem 2.7,  Propositions 3.1 and 3.4]{lh}.
\begin{lemma}\label{3.5}
The following statements hold:
\item[(1)]
$SHO'$ is spanned by $\mathcal{G}$.
\item[(2)]
$SHO'=\overline{SHO}\oplus\mathrm{span_{\mathbb{F}}}\{x^{(\alpha)}x^{u}\in SHO\mid
\mathbf{I}(\alpha, u)=\mathbf{\widetilde{I}}(\alpha,u)=\emptyset\},$
where
$$
\overline{SHO}=\mathrm{span_{\mathbb{F}}}\big\{\widetilde{\mathcal{G}}\cup\{x^{(\alpha)}x^{u}\in SHO\mid
\mathbf{I}(\alpha, u)=\emptyset,\mathbf{\widetilde{I}}(\alpha,u)\neq\emptyset\}\big\}.
$$
Moreover, $\overline{SHO}=\oplus_{i=-1}^{\xi-4}\overline{SHO}_i$.
\item[(3)]$SHO=\oplus_{i=-1}^{\xi-5}SHO_i$.
Moreover, $SHO_i=\overline{SHO}_i$.
\end{lemma}
As for Lemma \ref{th:2.1}(1), we have the following lemma by Lemma \ref{3.5}:
\begin{lemma}\label{3.6}
$SHO_1$ is an irreducible $SHO_0$-module.
\end{lemma}
\begin{proof}[\textit{Proof of Theorem \ref{03.1}(2)}] It is immediate by Lemmas \ref{1.2} and \ref{3.6}.
\end{proof}
\subsection{$KO(n)$}
Let $L=KO(n)$. Regard $KO(n)$ as a subalgebra of $KO(n+1)$.
To determine all the  MGS of type $(\textrm{I})$ for $L$,
for $i\in \mathbb{N}_0$, $\lambda \in \mathbb{F}$,
set
\begin{flalign}
L_{i,0}=&\Big\{f\in L_{i} \mid  f\in
\mathrm{Ann}(L_{-2})\Big\}=L_i\cap\mathcal{O}(n,n);\nonumber\\
L_{i,0}'=&\Big\{f\in L_{i,0} \mid
\Delta(f)=0\Big\}; \nonumber        \\
L_{i,1}=&\Big\{fx_{2n+1}\in L_i\mid  f\in
\mathrm{Ann}(L_{-2}) \Big \}; \nonumber\\
L_{i}^{\lambda}=&\{f\in L_i\mid \mathrm{\overline{div}}_{\lambda}(f)=0\},\nonumber
\end{flalign}
where $\mathrm{\overline{div}}_{\lambda}=\Delta+(\frak{D}-(\lambda+1)\mathrm{id_{\mathcal{O}}})\partial_{2n+1}$.
Note that
$$L_{i}^{\lambda}=\begin{cases}
 SKO''(n, (\lambda+1)/n)_i, &    ~\mbox{if}~ n\not\equiv 0 (\mathrm{mod} \  p),\\
SKO''(n+1,(\lambda+1)/(n+1))_i\cap KO(n),  &~\mbox{if}~ n\equiv 0 (\mathrm{mod} \  p).
\end{cases}$$
\textit{Hereafter for convenience, we put $z=x_{2n+1}$.}
As for Lemma \ref{th:2.1}, we have:
\begin{lemma}\label{th:2.5}
The following statements hold:
\begin{enumerate}
\item[(1)]$L_{1,0}'=U(L_{0}) D$, for $0\neq D\in L_{1,0}'$.
\item[(2)]$L_{1,0}=U(L_{0}) D$, for $D\in L_{1,0}$ and $\Delta(D)\neq 0$.
\item[(3)]$L_{1,1}=U(L_0)x_{i_0}z,$ for   $i_{0}\in\overline{1,2n}.$
\item[(4)]$L_{1}^{0}=U(L_{0})(x_{i_0}z)+U(L_{0})(x^{(3\varepsilon_{j_{0}})})$, for $i_{0}\in\overline{1,2n}$ and $j_0\in\overline{1,n}$.

If $\lambda\in \mathbb{F}\setminus \{0\}$, then
$$L_{1}^{\lambda}= U(L_{0})(x_{i_0}z+\lambda x_{i_{0}}x_{j_0}x_{{j_0}'}),$$
for $i_{0}\in\overline{1,2n}, j_{0}\in \overline{1,n}\backslash \{i_0\}.$
\end{enumerate}
\end{lemma}

\begin{lemma}\label{th:2.4}
The following statements hold:
\begin{enumerate}
\item[(1)]$L_{1,0}'$ and $L_{1,1}$
are irreducible $L_{0}$-submodules of $L_1$.
\item[(2)]$L_{1,0}$ is a maximal $L_{0}$-submodule of $L_{1}$.
\item[(3)]$L_{1}^{\lambda}$ is a maximal $L_{0}$-submodule of $L_{1}$, for any $\lambda\in \mathbb{F}$.
\end{enumerate}
\end{lemma}
\begin{proof}
(1) and (2) can be readily seen by Lemmas \ref{2.2} and \ref{th:2.5}. We only need to show (3).
Obviously  $L_{1}^{\lambda}$ is a proper $L_0$-submodule.
Suppose $M$ is an $L_{0}$-submodule of $L_{1}$ containing $L_{1}^{\lambda}$ properly.
Then there exists $D\in M\backslash L_{1}^{\lambda}$.
Put
$$\frak{A}  =\left\{x_{i}x_{2}x_{2'}\in L_{1} \mid i\in \overline{1,2n}\ \backslash\{{2,2'}\} \right\} \cup
\left\{x_{2}x_{1}x_{1'},x_{2'}x_{1}x_{1'}\right \}.$$
Since $\mathrm{\overline{div}}_{\lambda}$ is a surjective map from
$\mathrm{span}_{\mathbb{F}}\{\frak {A}\}$ onto $L_{-1}$,
there is $E \in \mathrm{span}_{\mathbb{F}}\{\frak{A}\}$
such that
$$
\mathrm{\overline{div}}_{\lambda}(E)=\mathrm{\overline{div}}_{\lambda}(D).
$$
As a result,
$E-D\in L_{1}^{\lambda}\subset M.$
Thus $0\neq E \in M$.
Moreover, by Lemmas \ref{2.2} and \ref{th3.2}, we may assume that
 $E=x_1x_2x_{2'}$. By Lemma \ref{th:2.5}, we can see that $L_{1,0}\subset M$.
 Moreover we have $M=L$.
Consequently,
$L_{1}^{\lambda}$ is a maximal $L_{0}$-submodule of $L_{1}$.
\end{proof}
\begin{proof}[\textit{Proof of Theorem \ref{03.1}(3)}]
Suppose
$$
\overline{M}=L_{-2}+L_{-1}+L_{0}+L_{1,0}+L_{2,0}+\cdots+L_{\xi-2,0}
$$
and
$$
M^{\lambda}=L_{-2}+L_{-1}+L_{0}+L^{\lambda}_{1}+\cdots+L^{\lambda}_{\xi-2},
$$
for $\lambda\in \mathbb{F}$.
Let $M$ be an MGS of type $(\textrm{I})$ for $L$.
Obviously, $M_1\neq 0$, since $L$ is transitive. Observe that
if there exists $\lambda\in \mathbb{F}$ such that $M_1\subset L_1^{\lambda}$,
then $M\subset M^{\lambda}$. Indeed, if $M\nsubseteq M^{\lambda}$,
then there exists $D\in M$ such that $\mathrm{\overline{div}}_{\lambda}(D)\neq 0$.
Further, there exists $D'\in M_1$ such that $\mathrm{\overline{div}}_{\lambda}(D')\neq 0$,
which contradicts to $M_1\subset L_1^{\lambda}$.
So $M\subset M^{\lambda}$.
From now on, for any $\lambda\in \mathbb{F}$, suppose $M_1\nsubseteq L_1^{\lambda}$.
Then for  $\lambda_0\in \mathbb{F}$,
 there exists $D\in M_1$ such that
$\mathrm{\overline{div}}_{\lambda_0}(D)\neq 0$.
 Assert that $D\in L_{1,0}$.
Assume conversely that  $D\notin L_{1,0}$.
Then, by Lemmas \ref{2.2} and \ref{th3.2}, their exists
$$D'=x_iz+a_1x_ix_1x_{1'}+\cdots+a_ix^{(2\varepsilon_i)}x_{i'}+\cdots+a_nx_ix_nx_{n'}\in M_1, $$
where  $a_j\in \mathbb{F}$, $i,j\in \overline{1,n}$.
If $a_1+\cdots+a_n\neq 0$, then
$D'\in L_1^{a_1+\cdots+a_n}$. By Lemma \ref{th:2.5}(4), we have
$L_1^{a_1+\cdots+a_n}\subset M_1.$
Thus by Lemma \ref{th:2.4}, we have $M_1=L_1^{a_1+\cdots+a_n}$, which is contradict to our assumption that  $M_1\nsubseteq L_1^{\lambda}$,
for any $\lambda\in \mathbb{F}$.
If $a_1+\cdots+a_n=0$, then $L_{1,1}\subset M$.
 We can see that $L_{1,1}\subsetneq M_1$, by our assumption. So there exists $0 \neq D''\in M_1\cap L_{1,0}'$, by Lemma \ref{th:2.5}(2) and the maximality of $M$.
Thus $L_1^0= M_1$, which is contradict to our assumption that  $M_1\nsubseteq L_1^{\lambda}$,
for any $\lambda\in \mathbb{F}$.
So $D\in L_{1,0}$. By Lemma \ref{th:2.5}, $L_{1,0}\subset M_1$, since $\mathrm{\overline{div}}_{\lambda_0}(D)\neq 0$.
Now we will show that $M \subset \overline{M}$. Assume there exists $\widetilde{D} \in M\backslash\ \overline M$.
Clearly, $\widetilde{D}$ can be written in the form $\widetilde{D}=u_{1}+u_{2}z$,
where $u_{1}, u_{2}\in \mathrm{Ann}(L_{-2})$ with
$u_2 \neq 0$.
Further,
there exists
$$
\widetilde{D}'=u'_1+u_2'z\in M_1\setminus\overline{M}_1,
$$
where $u'_1,  u'_2\in \mathrm{Ann}(L_{-2})$ with $u'_2\neq 0$.
Since $L_{1,0}\subset \overline{M} \bigcap M$,
we have $0 \neq u_{2}'z \in M\setminus L_{1,0}$.
Hence $M_1=L_1$. As a result,
$M=L$, which is contradict to the maximality of $M$. So $M=\overline{M}$.
Furthermore we have
$$
\{\overline{M}, M^\lambda \mid \lambda\in\mathbb{F}\}
$$
are all the MGS of type $(\textrm{I})$ for $L$.
\end{proof}
\subsection{$SKO(n,\lambda)$}
Let $L=SKO(n,\lambda)$.
Firstly, we introduce the following symbols for $q\in \widetilde{\mathbf{I}}(\alpha, u)$:
\begin{align}
A(\alpha,u,q)&=x^{(\alpha)}x^u-\sum_{i\in\mathbf{I}(\alpha,u)}\Gamma_i^q(x^{(\alpha)}x^u)\nonumber\\
B(\alpha,u,\lambda,q)&=(-1)^{|u|}(n\lambda-\mathrm{zd}(x^{(\alpha)}x^u))\nabla_q(x^{(\alpha)}x^u),\nonumber\\
E(\alpha,u)&=x^{(\alpha)}x^u,\nonumber\\
E(\alpha,u,2n+1)&=x^{(\alpha)}x^ux_{2n+1},\nonumber\\
E(\alpha,u,2n+1,q)&=x^{(\alpha)}x^ux_{2n+1}+B(\alpha,u,\lambda,q),\nonumber\\
G(\alpha,u,2n+1,q)&=A(\alpha,u,q)x_{2n+1}+B(\alpha,u,\lambda,q),\nonumber
\end{align}
and
\begin{align}
S_1&=\{E(\alpha,u)\mid \mathbf{I}(\alpha,u)=\emptyset,(\alpha,u)\neq(0,0)\},\nonumber\\
S_2&=\{A(\alpha,u,q)\mid x^{(\alpha)}x^u\in\mathcal{D}^*, q\in\widetilde{\mathbf{I}}(\alpha,u)\}, \nonumber\\
S_3&=\{E(\alpha,u,2n+1,q(\alpha,u))\mid\mathbf{I}(\alpha,u)=\emptyset,
\widetilde{\mathbf{I}}(\alpha,u)\neq \emptyset\},\nonumber\\
S_4&=\{G(\alpha,u,2n+1,q)\mid x^{(\alpha)}x^u\in\mathcal{D^*},
q\in\widetilde{\mathbf{I}}(\alpha,u)\}, \nonumber\\
S_5&=\{E(\alpha,u,2n+1)\mid\mathbf{I}(\alpha,u)=\emptyset,
\widetilde{\mathbf{I}}(\alpha,u)\neq \emptyset, n\lambda=\mathrm{zd}(x^{\alpha}x^u) ~\mbox{in}~ \mathbb{F}\},\nonumber
\end{align}
where $\mathbf{I}(\alpha, u)$, $\mathbf{\widetilde{I}}(\alpha, u)$ and $\mathcal{D^*}$ are defined as in Section 3.2, and $q(\alpha,u)=\min \widetilde{\mathbf{I}}(\alpha,u)$.
Put $G=G(\pi-\varepsilon_1,\langle2',\ldots, n'\rangle,2n+1,1)$, where $\pi=(\pi_1, \ldots, \pi_n)$, in which $\pi_i=p-1$ for $i\in\overline{1,n}$.
Given $r\in\overline{1,n}$, let
$$
\mathbf{J}(r)=\{(i_1,\ldots,i_r)\mid 1\leq i_1<\cdots<i_r\leq n)\}, \mathbf{J}(0)=\emptyset.
$$
For $(i_1,\ldots,i_r)\in \mathbf{J}(r)$, $r\in\overline{0,n}$, let
$$
X(i_1,\ldots,i_r)=x^{(\pi_{i_1}\varepsilon_{i_1}+\cdots+\pi_{i_r}\varepsilon_{i_r})}x^{\langle1',\ldots,n'\rangle-\langle i_{1'},\ldots,i_{r'}\rangle}.
$$
Given $\lambda\in \mathbb{F}$ and $l\in \mathbb{Z}$, put
$$
\frak{G}_l(\lambda,n)=\{k\in \overline{0,n}\mid n\lambda-n+2k+l=0\in \mathbb{F}\}.
$$
Put
$$\delta'_{n\lambda, -1}=
\begin{cases}
1, &  n\lambda+1=0 ~\mbox{in}~ \mathbb{F},\\
0, &  n\lambda+1\neq0 ~\mbox{in}~ \mathbb{F}.
\end{cases}$$
The following lemma follows from \cite[Theorems 2.6 and 2.7]{ly1}.
\begin{lemma} \label{3.11}
The following statements hold:
\item[(1)]$SKO''$ is spanned by $\{1\}$ and $\cup_{i=1}^5S_i$.
\item[(2)] If $n\lambda+1\neq 0 ~\mbox{in}~ \mathbb{F}$ or $\frak{G}_2(\lambda,n)=\emptyset$,
    $$
    SKO''=SKO'\oplus\mathrm{span_{\mathbb{F}}}S_5\oplus\sum_{r\in\frak{G}_2(\lambda,n)}\sum_{(i_1,\ldots,i_r)\in \mathbf{J}(r)}\mathbb{F}X(i_1,\ldots,i_r).
    $$
If $n\lambda+1= 0 ~\mbox{in}~ \mathbb{F}$ or $\frak{G}_2(\lambda,n)\neq\emptyset$,
$$
SKO''=SKO'\oplus\mathrm{span_{\mathbb{F}}}S_5\oplus\sum_{r\in\frak{G}_2(\lambda,n)}\sum_{(i_1,\ldots,i_r)\in \mathbf{J}(r)}\mathbb{F}X(i_1,\ldots,i_r)\oplus\mathbb{F}G.
$$
\item[(3)] $SKO'=SKO\oplus\mathbb{F}\delta'_{n\lambda,-1}G$. \qed
\end{lemma}
Recall that $SKO_{1,0}$ and $SKO_{1,1}$ are defined as in Theorem \ref{03.2}.
By Lemma \ref{3.11} and a direct calculation, we have the following lemma:
\begin{lemma}\label{3.10}
Suppose $L=SKO(n, \lambda)$.
\begin{enumerate}
\item[(1)]$L_{1,0}$
is the unique proper $L_0$-submodule of $L_1$, if $1+n\lambda\neq0$ in $\mathbb{F}$.
\item[(2)]$L_1$ has exactly two proper $L_0$-submodules $L_{1,0}$ and $L_{1,1}$, if $1+n\lambda=0$ in $\mathbb{F}$.
\end{enumerate}
\end{lemma}
\begin{proof}[\textit{Proof of Theorem \ref{03.1}(4)}] It is immediate by Lemma \ref{3.10}.
\end{proof}

\section{MGS of type (\MyRoman{2})}
Let $L$ be a Lie superalgebra of odd Cartan type.
In this section, we determine the MGS, $M$, of type (\MyRoman{2}) for $L$,
that is, $M_{-1}$ is a nontrivial subspace of $L_{-1}$.
Suppose $V$ is a nontrivial subspace of $L_{-1}$.
We introduce the following notions：
\begin{flalign}
 &M^{L}_{-1}(V)=V,                                                                      \nonumber \\
 &M^{L}_{-2}(V)=[M^{L}_{-1}(V),M^{L}_{-1}(V)],                                          \nonumber   \\
 &M^{L}_i(V)=\{u\in L_{i}\mid [u,V]\subset M^{L}_{i-1}(V)\}, \  \mbox{ for }i\geq0,     \nonumber \\
 &M^{L}(V)=\oplus_{i\geq -2}M^{L}_i(V).                                                  \nonumber
\end{flalign}
Without confusion, we usually omit the superscript $L$ of  $M^{L}_i(V)$ and $ M^{L}(V)$.
It is easy to verify the following lemma:
\begin{lemma}
Let $V$ be a nontrivial subspace of $L_{-1}$. Then
\begin{enumerate}
\item[(1)] $ M_0(V)$ is a subalgebra of $L_0$.
\item[(2)] $M_i(V)$ is an $M_0(V)$-module, for any $i\geq -2$.
\item[(3)]$[M_i(V),M_j(V)]\subset M_{i+j}(V)$, for any $ i,j\geq-2$.
\item[(4)] $M(V)$ is a subalgebra of $L$.\qed
\end{enumerate}
\end{lemma}
Suppose $L=KO(n)$ or $SKO(n,\lambda)$, and $V$ is an isotropic subspace of $L_{-1}$.
Set
\begin{align*}
&M_{-2}(L_{-2},V)=L_{-2},      \quad M_{-1}(L_{-2},V)=V, \\
&M_{i}(L_{-2},V)=\{u\in L_i\mid [u,V]\subset M_{i-1}(L_{-2},V), [u,L_{-2}]\subset M_{i-2}(L_{-2},V)\}.
\end{align*}
The following lemma can be verified directly:
\begin{lemma}\label{th:3.10}
Let $V$ be an isotropic subspace of $L_{-1}$.
Then
\begin{enumerate}
\item[(1)] $M_0(L_{-2},V)$ is a subalgebra of $L$.
\item[(2)] $M_i(L_{-2},V)$ is an $M_0(L_{-2},V)$-module, for $i\geqslant -2$.
\item[(3)] $[M_i(L_{-2},V),M_j(L_{-2},V)]\subset M_{i+j}(L_{-2},V)$, for $i,\ j\geqslant -2$.
\item[(4)] $M(L_{-2},V)=\oplus_{i\geq {-2}}M_i(L_{-2},V)$ is a subalgebra of $L$.
\end{enumerate}
\end{lemma}
Recall that $\mathrm{misdim}V$ is the dimension of maximal isotropic subspaces of a symplectic superspace $V$.
For $L=HO(n)$, $SHO(n)$, $KO(n)$ or $SKO(n,\lambda)$,
we put
\begin{align}
\mathcal{V}^L&=\{V\mid \mbox{$V$ is a nontrivial subspace of $L_{-1}$ with $\mathrm{misdim}V\neq n-1$}\};\nonumber\\
\mathfrak{V}_{\mathfrak{i}}^{L}&=\{V\in \mathcal{V}^L\mid \mbox{$V$ is an isotropic subspace of $L_{-1}$}\};\nonumber\\
\mathfrak{V}_{\mathfrak{n}}^{L}&=\{V\in \mathcal{V}^L\mid \mbox{$V$ is a nondegenerate subspace of $L_{-1}$}\}.\nonumber
\end{align}
For subspaces $V$ and $\widetilde{V}$ of $L_{-1}$,
$V\stackrel{\mathrm{sspace }}{\cong} \widetilde{V}$ means that
$V$ is sym-isotropic to $\widetilde{V}$ (Definition \ref{A6}).
For integers $k,l>0$ and $k \leq l$,
$$
\mathbf{J}(k,l)=\{(i_1,\ldots,i_k)\mid 1\leq i_1<\cdots< i_k\leq l\}.
$$
We have the following theorem:
\begin{theorem}\label{03.2}
All the  MGS of type (\MyRoman{2}) for Lie superalgebras of odd Cartan type are characterized as follows:

(1) For $L=HO(n)$ or $SHO(n)$,  we have:
 \begin{enumerate}
 \item[(a)\label{a}]All the  MGS of type (\MyRoman{2}) for $L$ are precisely
$$
\{M(V)\mid V\in \mathcal{V}^L\}.
$$
\item[(b)]
 For any $V$ and $\widetilde{V}$ in $\mathcal{V}^L$,
$$
M(V)\stackrel{\mathrm{alg}}{\cong} M(\widetilde{V}) \Longleftrightarrow V\stackrel{\mathrm{sspace }}{\cong} \widetilde{V}.
$$
\item[(c)]
$L$ has exactly $(n^3+3n^2+8n-6)/6$ pairwise  non-isomorphic type $(\mathrm{II})$ MGS.
\item[(d)]For $V\in \mathcal{V}^L$,
we have the following dimension formulas:
\begin{align}
\dim M^{HO}(V)=&2^np^n+2^kp^k-2^{s-t}p^{k+t}(2n-2s+1)-1;\nonumber\\
 \dim M^{SHO}(V)
=&\sum^{n}_{l=2}\bigg((2^{n-1}-2^{n-l})\sum_{(i_1,\ldots,i_l)\in\mathbf{J}(l,n)}\prod_{c=1}^l\pi_{i_c}\bigg)+(p+1)^n-2^n-2\nonumber\\
&-\Bigg(\bigg((p+2)^k+\sum^{k}_{m=2}\big((2^{k-1}-2^{k-m})\sum_{(j_1,\ldots,j_m)\in\mathbf{J}(m,k)}\prod_{c=1}^m\pi_{j_c}\big)\bigg)\cdot\nonumber\\
&\bigg(p^t2^{s-k-t}(2n-2s-1)-s+k\bigg)\Bigg)-\delta_{k,n-1}\delta_{s-t,n}2^{n-1}.
& \qedhere \nonumber
\end{align}
where $k=\mathrm{R}(V)/2$,
$s=\mathrm{misdim}V$ and
$t=\dim V_{\bar 0}-k$.
\end{enumerate}

(2) For $L=KO(n)$ or $SKO(n,\lambda)$, we have:
 \begin{enumerate}
 \item[(a)]
All the  MGS of type (\MyRoman{2}) for $L$ are precisely
$$
 \{M(V)\mid  V \in \mathcal{V}^L\}\cup \{M(L_{-2}, V)\mid V\in\mathfrak{V}_{\mathfrak{i}}^L\}.
$$
\item[(b)]
 For any $V$ and $\widetilde{V}$ in $\mathcal{V}^L$,
$$
M(V)\stackrel{\mathrm{alg}}{\cong} M(\widetilde{V}) \Longleftrightarrow V\stackrel{\mathrm{sspace }}{\cong} \widetilde{V}.
$$
For any  $V$ and $\widetilde{V}$ in $\mathfrak{V}_{\mathfrak{i}}^L$,
$$
 M(L_{-2}, V)\stackrel{\mathrm{alg}}{\cong} M(L_{-2}, \widetilde{V}) \Longleftrightarrow V\stackrel{\mathrm{sspace }}{\cong} \widetilde{V}.
$$
For $V\in \mathcal{V}^L$ and $\widetilde{V}\in \mathfrak{V}_{\mathfrak{i}}^L$,
$$
M(V)\ncong M(L_{-2}, \widetilde{V}).
$$
\item[(c)]
$L$ has exactly $(n^3+3n+20n-12)/6$ pairwise non-isomorphic type $(\mathrm{II})$ MGS.
\item[(d)] For $V\in \mathcal{V}^L$,
\begin{enumerate}
\item if $V$ is nonisotropic, then
\begin{align}
\dim M^{KO}(V)
=&2^{n+1}p^n+2^{k+1}p^k-2^{s-t+1}p^{k+t}(2n-2s+1);\nonumber
\end{align}
where $k=\mathrm{R}(V)/2$, $s=\mathrm{misdim}V$
and
$t=\dim V_{\bar 0}-k$.
\item if $V$ is isotropic, then
\begin{align}
\dim M^{KO}(V)=&2^{n+1}p^n-2^{s-t}p^{t}(2n-2s+1);\nonumber\\
\dim M(KO_{-2},V)
=&2^{n+1}p^n-p^{t}2^{s-t+1}(2n-2s+1)+2,\nonumber
\end{align}
where
$s=\dim V$ and
$t=\dim V_{\bar 0}$.\qed
\end{enumerate}
\end{enumerate}
\end{theorem}
We will prove Theorem \ref{03.2} (1) and (2) in the following subsections, respectively.
\subsection{$HO(n)$ and $SHO(n)$}
Suppose $L=HO(n)$ or $SHO(n)$.
We will determine all the MGS of type (\MyRoman{2}) for $L$.
Recall that $L_{-1}=\mathrm{span}_{\mathbb{F}}\{x_i\in L\mid i\in \overline{1,2n}\}$.
We call a subspace $V$ of $L_{-1}$ is standard if $V$ is spanned by certain standard basis elements
\begin{equation*}
\{x_i\in L_{-1}\mid i\in\overline{1,k} \cup \overline{1',k'} \cup
 \overline{k+1,k+t}\cup \overline{(k+t+1)',s'}\},
\end{equation*}
where $1\leq k\leq s \leq n$ and $0\leq t\leq s-k$.
Let $V$ be a standard subspace of $L_{-1}$ with $k$, $s$ and $t$ as above mentioned.
Obviously,  $k=\mathrm{R}(V)/2$,
$s=\mathrm{misdim}V$ and
$t=\dim V_{\bar 0}-k$.
We set subspaces $V_1$, $V_2$, $\overline{V}_2$ and $V_3$ associated with $V$ as follows:
\begin{equation}\label{4.2}
\begin{split}
V_1&=\mathrm{span}_{\mathbb{F}}\left\{x_i\in L_{-1}\mid i\in \mathbf{I}_{v_1}\right\}, \\
V_2&=\mathrm{span}_{\mathbb{F}}\left\{x_i\in L_{-1}\mid i\in \mathbf{I}_{v_2}\right\},\\
\overline{V}_2&=\mathrm{span}_{\mathbb{F}}\left\{x_i\in L_{-1}\mid i\in\mathbf{I}_{\overline v_2}\right\}, \\
V_3&=\mathrm{span}_{\mathbb{F}}\left\{x_i\in L_{-1}\mid i\in\mathbf{I}_{v_3}\right\},
\end{split}
\end{equation}
where
\begin{align}
\mathbf{I}_{v_1}&=\overline{1,k} \cup \overline{1',k'},\nonumber\\
\mathbf{I}_{v_2}&=\overline{k+1,k+t}\cup \overline{(k+t+1)',s'},\nonumber\\
\mathbf{I}_{\overline{v}_2}&=\overline{(k+1)',(k+t)'}\cup \overline{k+t+1,s},\nonumber\\
\mathbf{I}_{v_3}&=\overline{s+1,n} \cup \overline{(s+1)',n'}\nonumber
\end{align}
in which $
\mathrm{span}_{\mathbb{F}}\left\{x_i\mid i\in \emptyset\right\}=0
$.
For $l>0$, suppose $V_1, V_2,\ldots, V_l$ are standard subspaces of $L_{-1}$.
Set
$$\prod_{i=1}^l{V_{i}}=\mathrm{span}_{\mathbb{F}}\{x_{j_1}x_{j_2}\cdots x_{j_l}\in L_{l-2}\mid x_{j_i}\in V_{i}, i\in\overline{1,l}\}.$$
If $V_1=V_2=\ldots=V_l=V$, then we write $\prod_{i=1}^lV_i=V^l$, $V^0=\mathbb{F}$ and $V^{-l}=0$ for $l>0$.
Write  $V^0V^l=V^l$ for $l\in \mathbb{Z}$.

Suppose $V$ is a standard subspace of $HO_{-1}$.  Let $V_1,V_2,\overline{V}_2$ and $V_3$ be subspaces of $HO_{-1}$
associated with $V$ which are defined as in \eqref{4.2}.
Then we have the following lemma by induction:
\begin{lemma}\label{4.4}
Suppose $L=HO(n)$. Then we have:
\begin{enumerate}
\item[(1)]$M_0(V)=V_1^2\oplus V_2L_{-1}\oplus V_3^2.$
\item[(2)]$M_i(V)=V_1^{i+2}\oplus \left(\bigoplus_{j>0,j+k=i+2}V_2^jL_{-1}^k\right)\oplus \left(\bigoplus_{l>1,l+m=i+2} V_3^l\overline{V}_2^m\right)$, for $i>0$.
\item[(3)]$M(V)=\left(\bigoplus_{i>0}V_1^{i}\right)\oplus \left(\bigoplus_{j>0,k\geq0}V_2^jL_{-1}^k\right)\oplus \left(\bigoplus_{l>1,m\geq0} V_3^l\overline{V}_2^m\right)$.
\end{enumerate}
Next we still suppose $L=HO(n)$ or $SHO(n)$.
\begin{lemma}\label{th:3.4}
Suppose $V$ and $\widetilde{V}$ are nontrivial subspaces of $L_{-1}$.
Then
$$
M(V)\stackrel{\mathrm{alg}}{\cong} M(\widetilde{V}) \Longleftrightarrow V\stackrel{\mathrm{sspace }}{\cong} \widetilde{V}.
$$
\end{lemma}
\begin{proof}
Obviously, $M(V)\stackrel{\mathrm{alg}}{\cong} M(\widetilde{V})$,
if $V\stackrel{\mathrm{sspace }}{\cong} \widetilde{V}$ by Lemma \ref{th:1.3}.
Conversely, we may suppose $V$ and $\widetilde{V}$ are standard by Lemma \ref{th:1.3}. If there is a $\mathbb{Z}$-gradation algebraic isomorphism $\Phi: M(V)\longrightarrow M(\widetilde{V})$,
then $\Phi(V)=\widetilde V$ and $\Phi(M_0(V))=M_0(\Phi(V))=M_0(\widetilde{V})$.
Suppose  $\mathrm{R}(V)=2k$, $\mathrm{R}(\widetilde{V})=2k_1$,
$\mathrm{misdim}V=s$, $\mathrm{misdim}(\widetilde{V})=s_1$, $\dim V_{\bar 0}-k=t$
and $\dim \widetilde{V}_{\bar 0}-k_1=t_1$.
Since $(\dim V_{\bar 0}, \dim V_{\bar 1})= (\dim \widetilde {V}_{\bar 0}, \dim \widetilde{V}_{\bar 1})$
and
$$\dim M_0(V)-\dim M_0(\widetilde{V})=0,$$
we have
$ks=k_1s_1$, by Lemma \ref{4.4}.
The equation $k_1+s_1=k+s$ forces $k=k_1, \ s=s_1$ and $t=t_1$.
Thus $V\stackrel{\mathrm{sspace }}{\cong} \widetilde{V}$.
\end{proof}
\end{lemma}
Recall $\mathcal{V}^L$ and $\mathfrak{V}_{\mathfrak{n}}^L$ defined before Theorem \ref{03.2}.
We have the following lemma:
\begin{lemma}\label{th:3.3}
Let $V$ be a proper subspace of $L_{-1}$. Then
 $M(V)$ is an MGS of $L$ if and only if $V\in \mathcal{V}^L$.
\end{lemma}

\begin{proof}
Let $L=HO(n)$.
First we consider the sufficiency.
By Lemma \ref{th:3.4},
we may assume that $V$ is a standard subspace of $L_{-1}$.  Suppose $V_1,V_2,\overline{V}_2$ and $V_3$ are subspaces of $L_{-1}$
associated with $V$ which are defined as in \eqref{4.2}, and $\mathbf{I}_{v_1},\mathbf{I}_{v_2}, \mathbf{I}_{\overline{v}_2}$ and
$\mathbf{I}_{v_3}$ are corresponding index sets.
Obviously, $M(V)\subsetneq L$.
For any $u\in L\setminus M(V)$, set $$M=\langle {M(V), u \rangle}.$$
Notice that the standard torus $T$ is contained in $M(V)$. We can assume that
$
u=x_i\in L_{-1}\backslash V
$, by the definition of $M(V)$ and Lemma \ref{th3.2}.
We need to show that $M=L$.
If $V \in \mathfrak{V}_{\mathfrak{n}}^L$,
then $M_0(V)=V_1^2\oplus V_3^2$, by Lemma \ref{4.4}.
For $i'\neq j\in\mathbf{I}_{v_3}$, by Lemma \ref{2.2} we have
 $$L^{\epsilon_j}=[L^{\epsilon_i}, L^{\epsilon_{i'}+\epsilon_j}]\subset M.$$
Moreover, since $n-\mathrm{misdim}V>1$, for $i' \neq t\in\mathbf{I}_{v_3}$,
we have
$$L^{\epsilon_{i'}}=[L^{\epsilon_t}, L^{\epsilon_{t '}+\epsilon_{i'}}]\subset M.$$
Thus $L_{-1}\subset M.$
Since $M \supset M(M_{-1})=M(L_{-1})=L$, we have $M=L$.
For the remaining case that $V\in \mathcal{V}^L\backslash \mathfrak{V}_{\mathfrak{n}}^L$,
we still can show that $M(V)$ is an MGS of $L$ in the same way.

Next we consider the necessity.
Suppose
$V\in \mathcal{V}^L$.
We may assume that $V$ is a standard subspace of $L_{-1}$.
Suppose $V_1,V_2,\overline{V}_2$ and $V_3$ are subspaces of $L_{-1}$
associated with $V$ which are defined as in \eqref{4.2},  and $\mathbf{I}_{v_1},\mathbf{I}_{v_2}, \mathbf{I}_{\overline{v}_2}$ and
$\mathbf{I}_{v_3}$ are corresponding index sets.
Obviously $V_3=\{x_n\mid x_{n'}\}$ and $\mathbf{I}_{v_1}\cup \mathbf{I}_{v_2}\not=\emptyset$.
So there exists $ i\in \mathbf{I}_{v_1}\cup\mathbf{I}_{\bar v_2}$
such that $x_ix_n\not\in M(V)$, by Lemma \ref{4.4}.
We can see that  $M(V)\subsetneq \langle M(V), x_ix_n\rangle$ and  $x_ix_{n'}\not\in \langle M(V),x_ix_n\rangle$,
Consequence, $M(V)$ is not an MGS of $L$.

We can prove the lemma similarly for $L=SHO$, since $SHO_0\oplus x_1x_{1'}=HO_0$.
\end{proof}

\begin{proof}[\textit{Proof of Theorem \ref{03.2}(1)}]
Suppose $M$ is an MGS of type (\MyRoman{2}) for $L$.
It is known from induction that $M_i\subset M_i(M_{-1})$, for any $i\geq -1$.
Thus $M\subset M(M_{-1})$.
We have $M =M(M_{-1})$ by the maximality of $M$.
 Furthermore, (a), (b) and (c) of Theorem \ref{03.2}(1) hold by Lemmas \ref{th:3.4} and \ref{th:3.3}.
It remains to show the dimension formulas.
Let $L=HO(n)$. We may assume that $V$ is a standard subspace of $L_{-1}$.
Suppose $V_1,V_2,\overline{V}_2$ and $V_3$ are subspaces of $L_{-1}$
associated with $V$ which are defined as in \eqref{4.2}.
Put
$$
\overline{M}(V)=\Bigg(\bigoplus_{i\geq0,j>0}V_1^i\overline{V}_2^j\Bigg)\oplus \Bigg(\bigoplus_ {q\geq0,l\geq0}V_1^q\overline{V}_2^lV_3\Bigg).
$$
Then
$$\dim \overline {M}(V)=2^{s-t}p^{k+t}(2n-2s+1)-2^kp^k.$$
Further, we have
\begin{align}
\dim M^{HO}(V)&=\dim HO-\dim \overline{M}(V)\nonumber\\
&=2^np^n+2^kp^k-2^{s-t}p^{k+t}(2n-2s+1)-1.\nonumber
\end{align}

Next we consider $L=SHO$.
Write $\mathrm{codim}M^{SHO}(V)$ for the codimension of $M^{SHO}(V)$ in $SHO$.
Since
$$M^{SHO}(V)=M^{HO}(V)\cap SHO,$$
by Lemma \ref{3.5}, we have
\begin{align}
\mathrm{codim}M^{SHO}(V)=&\dim (\overline {M}(V)\cap SHO)\nonumber\\
=&
\Bigg(\bigg((p+2)^k+\sum^{k}_{m=2}\Big((2^{k-1}-2^{k-m})+\sum_{(j_1,\ldots,j_m)\in\mathbf{ J}(m,k)}\prod_{c=1}^m\pi_{j_c}\Big)\bigg)\cdot\nonumber\\
&\bigg(p^t2^{s-k-t}(2n-2s-1)-s+k\bigg)\Bigg)-\delta_{k,n-1}\delta_{s-t,n}2^{n-1}.\nonumber\\
\nonumber
\end{align}
Thus,
\begin{align}
 \dim M^{SHO}(V)=&\dim SHO-\mathrm{codim}M^{SHO}(V)\nonumber\\
=&\sum^{n}_{l=2}\bigg((2^{n-1}-2^{n-l})\sum_{(i_1,\ldots,i_l)\in\mathbf{J}(l,n)}\prod_{c=1}^l\pi_{i_c}\bigg)+(p+1)^n-2^n-2\nonumber\\
&-\Bigg(\bigg((p+2)^k+\sum^{k}_{m=2}\big((2^{k-1}-2^{k-m})\sum_{(j_1,\ldots,j_m)\in\mathbf{J}(m,k)}\prod_{c=1}^m\pi_{j_c}\big)\bigg)\cdot\nonumber\\
&\bigg(p^t2^{s-k-t}(2n-2s-1)-s+k\bigg)\Bigg)-\delta_{k,n-1}\delta_{s-t,n}2^{n-1}.
& \qedhere \nonumber
\end{align}
\end{proof}


\subsection{$KO(n)$ and $SKO(n,\lambda)$}
Let $L=KO(n)$ or $SKO(n,\lambda)$. Put $z=x_{2n+1}$.
Assume that $HO(n)$ is embedded (as vector space) into $KO(n)$ naturally.
Let $M$ be an MGS of type (\MyRoman{2}) for $L$.
We first consider whether $M=M(M_{-1})$.
The following lemma is obvious.
\begin{lemma}\label{th:3.6}
Suppose $V$ is a nonisotropic subspace of $KO_{-1}$. Then
\begin{equation*}
M^{KO}_{i}(V)=\begin{cases} M^{HO}_{0}(V)\oplus \mathbb{F}z,& i=0,\\
M^{HO}_{i}(V)\oplus M^{HO}_{i-2}(V)z,& i\geq1.\\
    \end{cases}
\end{equation*}
\end{lemma}

\begin{proof}
First, let us show  ``$\supset$''.
This is obvious for $i=0$.
Suppose $i>0$.
We only need to show that
$$M^{HO}_{i-2}(V)z\subset M^{KO}_i(V),$$
since $M^{HO}_{i}(V)=M^{KO}_i(V)\cap HO$.
Take  $0\neq f\in M^{HO}_{i-2}(V)$ and $ v\in V$. Then
we get
$[f,v]z\in M^{KO}_{i-1}(V)$ by induction.
So $[fz, v]\in M^{KO}_{i-1}(V)$.
Thus $ fz \in M^{KO}_i(V)$.

Next we show  ``$\subset$''.
For any $i\in \mathbb{N}$, take
$$
0\neq u=u_0+u'z \in M^{KO}_i(V),
$$
where $ u_0$, $u' \in \mathrm{Ann}(KO_{-2})$.
For $i=0$, as $z\in M^{KO}_0(V)$,
we have
$
u_0 \in M^{HO}_0(V).
$
Suppose $i>0$.
If $u'=0$, then $u\in M^{HO}_i(V)$.
If $u'\not=0$,
we have
$$
-(-1)^{|u'|}2u'=[1,u]\in M^{HO}_{i-2}(V),
$$
since $V$ is a nonisotropic subspace.
As a result $ u' \in M^{HO}_{i-2}(V)$.
Moreover,
we have
$
u_0 \in M^{KO}_i(V),
$
since $u'z \in M^{KO}_i(V)$.
Consequently, ``$\subset $'' holds.
\end{proof}

\begin{lemma}\label{th3.8}
Suppose $V$ is  a nonzero isotropic subspace of $KO_{-1}$.
Then
\begin{equation*}
M^{KO}_{i}(V)=\begin{cases} M^{HO}_{0}(V)\oplus \mathbb{F}z,& i=0,\\
M^{HO}_{i}(V)\oplus HO_{i-2}z,& i\geq1.\\
    \end{cases}
\end{equation*}
\end{lemma}
\begin{proof}
First let us show ``$ \supset $''.
This is obvious for $i=0$. Suppose $i>0$.
Note that  $M^{HO}_{i}(V)\subset M^{KO}_i(V)$.
Thus we have to show that $HO_{i-2}z\subset M^{KO}_i(V).$
For any $0\neq f\in HO_{i-2}$ and $v\in V$,
we see that
$[f,v]z\in M^{KO}_{i-1}(V)$
by induction.
Moreover,
we have $fv \in M^{HO}_{i-1}(V)$, since
$V$ is an isotropic subspace of $KO_{-1}$.
So
$
[v, fz]\in M^{KO}_{i-1}(V).
$
Thus $ fz \in M^{KO}_i(V)$.
Then the inclusion  ``$\subset $'' is obvious.
\end{proof}

\begin{lemma}\label{th:3.7}
Let $V$ be a proper subspace of $L_{-1}$. Then
$M(V)$ is an MGS of type (\MyRoman{2}) for $L$
if and only if $V\in\mathcal{V}^L$.
\end{lemma}

\begin{proof}
It is immediate from the proof of Lemma \ref{th:3.3}.
\end{proof}
The following lemma is a direct consequence of Lemma \ref{th:3.7}.
\begin{lemma}
Let $M$ be an MGS of type (\MyRoman{2}) for $L$ and $[M_{-1}, M_{-1}]=M_{-2}$.
 Then $M=M(M_{-1})$.
\end{lemma}
\begin{remark}
If $[M_{-1}, M_{-1}]\neq M_{-2}$, then $M \neq M(M_{-1})$.
\end{remark}
Next we determine MGS, $M$, of type (\MyRoman{2}) for $L$ with $[M_{-1}, M_{-1}]\neq M_{-2}$.
Let $V$ be an isotropic subspace of $ L_{-1}$.
Recall that
\begin{align}
&M_{-2}(L_{-2},V)=L_{-2}, \quad M_{-1}(L_{-2},V)=V,     \nonumber    \\
&M_{i}(L_{-2},V)=\{u\in L_i\mid [u,V]\subset M_{i-1}(L_{-2},V), [u,L_{-2}]\subset M_{i-2}(L_{-2},V)\}. \nonumber
\end{align}
\begin{lemma}\label{th:3.11}
Let $V$ be an isotropic subspace of $KO_{-1}$.
Then
\begin{equation*}
M_i(KO_{-2},V)=\begin{cases} M^{HO}_0(V)\oplus \mathbb{F}z,& i=0,\\
M^{HO}_{i}(V)\oplus M^{HO}_{i-2}(V)z,& i\geq1.\\
    \end{cases}
\end{equation*}
\end{lemma}
\begin{lemma}\label{th:3.12}
$M(L_{-2},V)$ is an MGS of $L$ if and only if $V\in \mathfrak{V}_{\mathfrak{i}}^L$.
\end{lemma}
\begin{proof}[\textit{Proof of Theorem \ref{03.2}(2)}]
Let $M$ be an MGS of type  (\MyRoman{2})  for $L=KO(n)$ or $SKO(n,\lambda)$.
Suppose $V=M_{-1}$. From Lemma \ref{th:1.3},
there is an automorphism
$\Phi\in L$ such that $\Phi(\widetilde{V})=V$,
where $\widetilde{V}$ is a standard subspace  of $L_{-1}$.
We have
$$
\Phi(M(\widetilde{V}))=M(\Phi(\widetilde{V}))=M(V).
$$
If $V=M_{-1}$ is nonisotropic, then
$
\Phi(M(\widetilde{V}))=M
$
by $M\subset M(V)$.
If $V=M_{-1}$ is isotropic and $M_{-2}=0$, then $M(V)=\Phi(M(\widetilde{V}))=M$.
If $V=M_{-1}$ is isotropic and $M_{-2}\neq0$, then $M(L_{-2},V)=\Phi(M(L_{-2}, \widetilde{V}))=M$.
Furthermore, (a), (b) and (c) of Theorem \ref{03.2}(2) is obvious by Lemmas \ref{th:3.7} and \ref{th:3.12}.
Next we calculate the dimensions of $M^{KO}(V)$ and $M(KO_{-2}, V)$ separately.
If $V\in\mathcal{V}^{KO}$ and $V$ is nonisotropic, then by Lemma \ref{th:3.6}, we have
$$
\dim M^{KO}(V)=2(\dim M^{HO}(V)+1)
=2^{n+1}p^n+2^{k+1}p^k-2^{s-t+1}p^{k+t}(2n-2s+1),
$$
where $k=\mathrm{R}(V)/2$,
$s=\mathrm{misdim}V$ and $t=\dim V_{\bar 0}-k$.
If $V\in\mathfrak{V}_{\mathfrak{i}}^{KO}$, by Lemmas \ref{th3.8} and  \ref{th:3.11}, we have
\begin{align}
\dim M^{KO}(V)&=\dim M^{HO}(V)+p^n2^n=2^{n+1}p^n-2^{s-t}p^{t}(2n-2s+1);\nonumber\\
\dim M(KO_{-2},V)&=2(\dim M^{HO}(V)+1)=2^{n+1}p^n-2^{s-t+1}p^t(2n-2s+1)+2\nonumber
\end{align}
where $s=\mathrm{misdim}V$ and $t=\dim V_{\bar 0}$.
\end{proof}

\section{MGS  of Type (\textrm{III})}
Suppose $L$ is a Lie superalgebra of odd Cartan type.
We discuss the MGS, $M$, of type (\MyRoman{3}) for $L$,
that is,  $M_{-1}=L_{-1} ~\mbox{ while }~ M_{0}\neq L_{0}$.
Let $A_{0}$ be a subalgebra of $L_{0}$.
If $L_{-1}$ is an irreducible $A_{0}$-module,
then we say that $A_{0}$ is irreducible;
otherwise we say that $A_{0}$ is reducible.
\begin{definition}
A graded subalgebra $A$
of $L$
is called an R-subalgebra (resp. S-subalgebra)
if $A_{0}$ is a reducible (resp. irreducible) subalgebra in $L_0$.
\end{definition}
Let $A_{0}$ be a nontrivial subalgebra of $L_{0}$.
We define a subalgebra of $L$ as follows:
\begin{equation*}
 M(L_{-1},A_{0})=\oplus_{i\geq-2}M_{i}(L_{-1},A_{0}),
\end{equation*}
where
$$
M_{i}(L_{-1},A_{0})
=\begin{cases} L_{i}, & i<0,\\
A_{0}, & i=0,\\
\left\{u\in L_{i}\mid [u,L_{-1}]\subset M_{i-1}(L_{-1},A_{0})\right\}, & i>0.\\
    \end{cases}
$$
Recall that
\begin{align}
\mathcal{V}^L&=\{V\mid \mbox{$V$ is a nontrivial subspace of $L_{-1}$ with $\mathrm{misdim}V\neq n-1$}\};\nonumber\\
\mathfrak{V}_{\mathfrak{i}}^{L}&=\{V\in \mathcal{V}^L\mid \mbox{$V$ is an isotropic subspace of $L_{-1}$}\};\nonumber\\
\mathfrak{V}_{\mathfrak{n}}^{L}&=\{V\in \mathcal{V}^L\mid \mbox{$V$ is a nondegenerate subspace of $L_{-1}$}\}.\nonumber
\end{align}
For integers $k,l>0$ and $k \leq l$,
$$
\mathbf{J}(k,l)=\{(i_1,\ldots,i_k)\mid 1\leq i_1<\cdots< i_k\leq l\}.
$$
We have the following theorem:
\begin{theorem}\label{0.003}
All the maximal R-subalgebras of type (\MyRoman{3}) for Lie superalgebras of odd Cartan type are characterized as follows:

(1) For $L=HO(n)$ or $SHO(n)$, we have:
\begin{enumerate}
\item[(a)]    All the maximal R-subalgebras of type (\MyRoman{3}) for $L$ are precisely
$$
\{M(L_{-1}, M_0(V))\mid V\in \mathfrak{V}_{\mathfrak{n}}^{L}\cup \mathfrak{V}_{\mathfrak{i}}^{L}\}.
$$
\item[(b)]
 For any  $V$ and $\widetilde{V}$ in $\mathfrak{V}_{\mathfrak{i}}^{L}$,
 $$
M(L_{-1}, M_0(V))\stackrel{\mathrm{alg}}{\cong} M(L_{-1}, M_0(\widetilde{V}) )\Longleftrightarrow
V\stackrel{\mathrm{sspace }}{\cong} \widetilde{V}.
$$
For any $V$ and $\widetilde{V}$ in $\mathfrak{V}_{\mathfrak{n}}^{L}$,
\begin{equation*}
M(L_{-1}, M_0(V))\stackrel{\mathrm{alg}}{\cong} M(L_{-1}, M_0(\widetilde{V}) )\Leftrightarrow
\dim V+\dim \widetilde{V}=2n ~\mbox{or}~\dim V=\dim \widetilde{V}.
\end{equation*}
For any $V\in\mathfrak{V}_{\mathfrak{n}}^{L}$ and $\widetilde{V}\in\mathfrak{V}_{\mathfrak{i}}^{L}$,
$$
M(L_{-1}, M_0(V))\stackrel{\mathrm{alg}}{\ncong} M(L_{-1}, M_0(\widetilde{V})).
$$
\item[(c)]
$L$ has exactly
$\lfloor(n^2+2n-2)/2\rfloor$ pairwise non-isomorphic maximal R-subalgebras of type  (\MyRoman{3}).
\item[(d)]
We have the following dimension formulas:
\begin{enumerate}
\item if $V \in \mathfrak{V}_{\mathfrak{n}}^{L}$, then
\begin{align}
&\dim M(HO_{-1}, M_0(V))=(2p)^{k}+(2p)^{n-k}-2;\nonumber\\
&\dim M(SHO_{-1}, M_0(V))\nonumber\\
=&\sum^{k}_{l=2}\bigg(\big(2^{k-1}-2^{k-l}\big)\sum_{(i_1,\ldots,i_l)\in\mathbf{J}(l,k)}\prod_{c=1}^l\pi_{i_c}\bigg)+(p+1)^k\nonumber\\
&+\sum^{n-k}_{m=2}\bigg((2^{n-k-1}-2^{n-k-m})\sum_{(j_1,\ldots,j_m)\in\mathbf{J}(m,n-k)}\prod_{c=1}^m\pi_{j_c}\bigg)+(p+1)^{n-k}-2.\nonumber
\end{align}
\item if $V \in \mathfrak{V}_{\mathfrak{i}}^{L}$, then
\begin{align}
&\dim M(HO_{-1}, M_0(V))=p^{n-s+t}2^{n-t}+p^{t}2^{s-t}s-1;\nonumber\\
&\dim M(SHO_{-1}, M_0(V))\nonumber\\
=&p^t2^{n-s}\bigg(\sum^{n-s}_{l=2}\bigg(\big(2^{n-s-1}-2^{n-s-l}\big)\sum_{(i_1,\ldots,i_l)\in\mathbf{J}(l,n-s)}\prod_{c=1}^l\pi_{i_c}\bigg)+(p+1)^{n-s}-1\bigg)\nonumber\\
&-2^{n-s}+(t-1)2^sp^{t-2}(p-1)^2+(s-1)p^t2^{s-2}+p^{t-1}(p-1)2^{s-1},\nonumber
\end{align}
where $s=\mathrm{misdim}V$
and
$t=\dim V_{\bar 0}-k$.
\end{enumerate}
\end{enumerate}

(2) For $L=KO(n)$ or $SKO(n,\lambda)$, we have:
\begin{enumerate}
\item[(a)]
All the maximal R-subalgebras of type (\MyRoman{3}) for $L$  are precisely
$$
 \{M(L_{-1}, M_0(V)) \mid V \in \mathfrak{V}_{\mathfrak{i}}^L\}.
$$
\item[(b)]
 For any  $V$ and $\widetilde{V}$ in $\mathfrak{V}_{\mathfrak{i}}^L$,
$$
M(L_{-1}, M_0(V))\stackrel{\mathrm{alg}}{\cong} M(L_{-1}, M_0(\widetilde{V}) )\Longleftrightarrow
V\stackrel{\mathrm{sspace }}{\cong} \widetilde{V}.
$$
\item[(c)]
$L$ has exactly $n(n+1)/2$ pairwise non-isomorphic maximal R-subalgebras of type  (\MyRoman{3}).
\item[(d)]For any $ V \in \mathfrak{V}_{\mathfrak{i}}^L$,
we have the following dimension formulas:
 $$
\dim M(KO_{-1}, M_0(V))=\begin{cases}p^{n-s+t}2^{n-t+1}+p^t2^{s-t}s, &\ s\neq n.\\p^t2^{n-t+1}(1+n), &  \ s=n.
\end{cases}
$$
where $s=\mathrm{misdim}V$
and
$t=\dim V_{\bar 0}$.
\end{enumerate}
\end{theorem}

\begin{lemma}\label{th4.1}
Let $M$ be an MGS of type (\MyRoman{3}) for $L$.
Then $M_{0}$ is a maximal subalgebra of $L_{0}$.

\end{lemma}
\begin{proof}
Suppose $M_{0}$ is not a maximal subalgebra of $L_{0}$,
and $\overline{M}_{0}$ is a proper subalgebra of $L_{0}$,
which contains $M_{0}$ properly.
For $i\in \mathbb{N}$, define
\begin{equation*}
\overline{M}_{i}=\{u\in L_{i}\mid [u,L_{-1}]\subset \overline{M}_{i-1}\}.
\end{equation*}
Using induction on $i$,
we can see that $\overline{M}_{i}$ is an $\overline{M}_{0}$-module,
and $M_{i}\subset \overline{M}_{i}$.
So $L_{-2}+L_{-1}+\overline{M}_{0}+\overline{M}_{1}+\overline{M}_{2}+\cdots$
is a proper subalgebra of $L$ containing $M$ properly,
which contradicts to the maximality of $M$.
This shows that $M_{0}$ must be a maximal subalgebra of $L_{0}$.
\end{proof}
\begin{lemma}\label{th:4.2}
If $M$ is an MGS of type  (\MyRoman{3}) for $L$,
then there exists a maximal subalgebra $A_{0}$  of $L_{0}$
such that $M=M(L_{-1}, A_{0})$.
\end{lemma}

\begin{proof}
Suppose $M$ is an MGS of type (\MyRoman{3}) for $L$.
Then $M_{0}$ is a maximal subalgebra of $L_{0}$ by Lemma \ref{th4.1}.
Let $A_{0}=M_{0}$.
If we set $M_{i-1}\subset M_{i-1}(L_{-1},A_{0})$ for any $i\geq0$,
then we can get $M_{i}\subset M_{i}(L_{-1},A_{0})$
from the fact $[M_{i}, L_{-1}]\subset M_{i-1}$.
As a result, $M\subset M(L_{-1},A_{0})$.
Consequently, $M=M(L_{-1},A_{0})$ by the maximality of $M$.
\end{proof}
\begin{lemma}\label{th:4.4}
Suppose $A_{0}$ is a  reducible maximal subalgebra of $L_{0}$.
Then there exists a proper subspace $V$ of $L_{-1}$
such that $A_{0}=M_{0}(V)$.
\end{lemma}
\begin{proof}
Since $L_{-1}$ is a reducible $A_{0}$-module,
we can set $V$ to be a nontrivial subspace of $L_{-1}$ which is invariant under $A_{0}$.
So $A_{0}\subset M_{0}(V)$. Notice that $M_0(V)\neq L_0$, since $L_{-1}$ is an irreducible-$L_0$ module.
Thus
we have $A_{0}=M_{0}(V)$.
\end{proof}
\subsection{$HO(n)$ and $SHO(n)$}
Let $L=HO(n)$ or $SHO(n)$.
We will determine the maximal R-subalgebras of type (\MyRoman{3}) for $L$.
\begin{proposition}\label{th:4.5}
Suppose $A_{0}$ is a reducible maximal subalgebra of
$L_{0}$. Then $M(L_{-1},A_{0})$ is an MGS of $L$.
\end{proposition}
\begin{proof}
For any
$$
0\neq D \in L\backslash M(L_{-1},A_{0}),
$$
there exist $u_{1},\ldots,u_{l}\in L_{-1}$
such that
\begin{equation*}
   0\neq[u_{l},\ldots,[u_{1}, D]\ldots]=D'\in L_{0}\backslash A_{0}.
\end{equation*}
Since $A_{0}$ is maximal in $L_{0}$, we have
\begin{equation*}
 L_{0}\subset \langle{M_{0}(L_{-1},A_{0}), D'\rangle}\subset \langle {M(L_{-1},A_{0}),D\rangle}.
\end{equation*}
As $A_0$ is a reducible subalgebra of $L_{0}$,
there exists $V\subset L_{-1}$
such that $A_{0}=M_{0}(V)$, by Lemma \ref{th:4.4}.
We may assume that $V$ is a standard subspace of $L_{-1}$.
Suppose $V_1,V_2,\overline{V}_2$ and $V_3$ are subspaces of $L_{-1}$
associated with $V$ which are defined as in \eqref{4.2}, and $\mathbf{I}_{v_1},\mathbf{I}_{v_2}, \mathbf{I}_{\overline{v}_2}$ and
$\mathbf{I}_{v_3}$ are corresponding index sets.
Obviously, $M_{1}(L_{-1},A_{0})\neq 0$.
We assert that  $L_{1}\subset\langle M(L_{-1}, A_{0}), D\rangle$.
This is because that
$SHO_1$ is an irreducible $SHO_0$-module;
and that
for $L=HO$,
if $\mathbf{I}_{v_1}\neq \emptyset$,
then there exists $ i\in \mathbf{I}_{v_1} $ such that
$x^{(2\varepsilon_{i})}x_{i'}\in M_{1}(L_{-1},A_{0});$
if $\mathbf{I}_{v_1}\neq \emptyset$,
there exist $i\in \mathbf{I}_{v_2}$ and $i'\neq j\in\mathbf{I}_{v_2}\cup\mathbf{I}_{v_3}$  such that
$x_ix_jx_{j'}\in M_{1}(L_{-1},A_{0})$. By Lemma \ref{th:2.1}, we have $L_{1}\subset\langle M(L_{-1}, A_{0}), D\rangle$.
Thus $\langle M(L_{-1}, A_{0}), D\rangle=L$.
This completes the proof.
\end{proof}
As for $\mathfrak{V}^{L}_{\mathfrak{n}}$ and $\mathfrak{V}^{L}_{\mathfrak{i}}$, we define
\begin{align}
\mathfrak{V}^{L}_{\mathfrak{ne}}=&\{V\in\mathcal{V}^L\setminus \{\mathfrak{V}^{L}_{\mathfrak{n}}\cup\mathfrak{V}^{L}_{\mathfrak{i}}\}\mid V \mbox{  has  an $n$-dimensional }\nonumber\\
& \mbox{ isotropic subspace  and  $\mathrm{R}(V) \neq 2$}\}.\nonumber
\end{align}

\begin{lemma}\label{th:4.6}
Let $V$ be a subspace of $L_{-1}$.
Then $M_0(V)$ is a maximal subalgebra of $L_{0}$ if and only if $V\in \mathfrak{V}^{L}_{\mathfrak{n}}\cup\mathfrak{V}^{L}_{\mathfrak{i}}\cup\mathfrak{V}^{L}_{\mathfrak{ne}}$.
\end{lemma}

\begin{proof}
Suppose $L=HO(n)$.
We may assume that $V$ is a standard subspace of $L_{-1}$.
Suppose $V_1,V_2,\overline{V}_2$ and $V_3$ are subspaces of $L_{-1}$
associated with $V$ which are defined as in \eqref{4.2},  and $\mathbf{I}_{v_1},\mathbf{I}_{v_2}, \mathbf{I}_{\overline{v}_2}$ and
$\mathbf{I}_{v_3}$ are corresponding index sets.
First we consider the sufficiency.
We need to show that $M_0(V)$ is a maximal subalgebra of $L_0$, for any $V\in \mathfrak{V}^{L}_{\mathfrak{n}}\cup\mathfrak{V}^{L}_{\mathfrak{i}}\cup\mathfrak{V}^{L}_{\mathfrak{ne}}$.
If $V\in \mathfrak{V}^{L}_{\mathfrak{i}}$,
then $L_{-1}=V_2\oplus \overline{V}_2\oplus V_3$, where $V_2=V$.
 By Lemma \ref{4.4}, we have $M_0(V)=VL_{-1}\oplus V_3^2$ and $L/M_0(V)\cong \overline{V}_2^2\oplus \overline{V}_2V_3$.
Assert that  $\overline{V}_2^2\oplus \overline{V}_2V_3$ is an irreducible $M_0(V)$-module.
In fact, suppose $N$ is a nonzero $M_0(V)$-submodule of $\overline{V}_2^2\oplus \overline{V}_2V_3$.
Since the standard torus $T\subset M_0(V)$, there is $0\neq x_ix_j\in N$, by Lemma \ref{th3.2}.
If $(i,j)\in \mathbf{I}_{\overline{v}_2}\times\mathbf{I}_{\overline{v}_2}$, then for any $(q,l)\in \mathbf{I}_{\overline{v}_2}\times\mathbf{I}_{\overline{v}_2}\cup\mathbf{I}_{\overline{v}_2}\times\mathbf{I}_{v_3}$,
we have
\begin{align}
L^{\epsilon_q+\epsilon_l}&\subset[L^{\epsilon_i+\epsilon_j}, L^{\epsilon_{j'}+\epsilon_l}]\subset N, &~\mbox{for}~ q=i;\nonumber\\
L^{\epsilon_q+\epsilon_l}&\subset[[L^{\epsilon_q+\epsilon_{j'}},L^{\epsilon_i+\epsilon_j}],L^{\epsilon_{i'}+\epsilon_l}]\subset N, &~\mbox{for}~ q\neq i.\nonumber
\end{align}
Thus, $N=\overline{V}_2^2\oplus \overline{V}_2V_3$.
If $(i,j)\in \mathbf{I}_{\overline{v}_2}\times \mathbf{I}_{v_3}$, we may  assume that  $t\in\{(s+1)',\ldots, n'\}$,
since $n-s>1$.
Then $L^{\epsilon_i+\epsilon_{j'}}=[L^{\epsilon_i+\epsilon_j}, L^{\epsilon_{j'}+\epsilon_{j'}}]\in N$.
For $l\in\mathbf{I}_{\overline{v}_2}\cup\mathbf{I}_{v_3}\backslash\{j,j'\}$,  we have
$$
L^{\epsilon_i+\epsilon_l}\subset[L^{\epsilon_i+\epsilon_{j'}},[L^{\epsilon_i+\epsilon_{j}}, L^{\epsilon_{i'}+\epsilon_l} ]]\subset N.
$$
Furthermore,
for $(q,l)\in\mathbf{I}_{\overline{v}_2}\times\mathbf{I}_{\overline{v}_2}\cup\mathbf{I}_{\overline{v}_2}\times\mathbf{I}_{v_3}$,
we have
$$L^{\epsilon_q+\epsilon_l}\subset[L^{\epsilon_i+\epsilon_l}, L^{\epsilon_q+\epsilon_{i'}}]\subset N.$$
Thus, $N=\overline{V}_2^2\oplus \overline{V}_2V_3$.
So $M_0(V)$ is a maximal subalgebra of $L_0$.
Similarly, for $V\in \mathfrak{V}_{\mathfrak{n}}^{L}\cup \mathfrak{V}_{\mathfrak{ne}}^{L}$, the fact that $L/M_0(V)$ is an irreducible $M_0(V)$-module implies $M_0(V)$
is a maximal subalgebra of $L_0$.
For $L=SHO$, the lemma can be shown in a similar way, since $M^{SHO}_0(V)\oplus\mathbb{F}x_1x_{1'}=M^{HO}_0(V)$.

Next we consider the necessity. We need to show that $M_0(V)$ is not a maximal subalgebra of $L_0$
if $V\not\in \mathfrak{V}_{\mathfrak{n}}^{L}\cup\mathfrak{V}_{\mathfrak{i}}^L\cup \mathfrak{V}_{\mathfrak{ne}}^{L}$.
By Lemma \ref{th:3.4}, we may assume that $V$ is a standard subspace of $L_{-1}$.
Suppose $V_1,V_2,\overline{V}_2$ and $V_3$ are subspaces of $L_{-1}$
associated with $V$ which are defined as in \eqref{4.2},  and $\mathbf{I}_{v_1},\mathbf{I}_{v_2}, \mathbf{I}_{\overline{v}_2}$ and
$\mathbf{I}_{v_3}$ are corresponding index sets.
For $\mathbf{I}_{v_1}\neq \phi $, $\mathbf{I}_{v_2}\neq \phi $ and $\mathbf{I}_{v_3}\neq \phi$,
there exists $(i,j)\in \mathbf{I}_{v_1}\times \mathbf{I}_{v_3}$ such that
$M_0(V)\subsetneq\langle M_0(V), x_ix_j\rangle$. However $\langle M_0(V), x_ix_j\rangle$ does not contain $\{x_qx_l\in L_0\mid (q,l)\in \mathbf{I}_{v_1}\times\mathbf{I}_{\overline{v}_2}\}$.
So $M_0(V)$ is not a maximal subalgebra of $L_0$. Similarly, for the remaining cases, we can show that $M_0(V)$ is not a maximal subalgebra of $L_0$.
\end{proof}

\begin{proof}[\textit{Proof of Theorem \ref{0.003}(1)}]
Let $L=HO(n)$.
Suppose $M$ is a maximal R-subalgebra of type (\MyRoman{3}) for $L$.
It is seen from Lemma \ref{th:4.4} that
there exists $V\subset L_{-1}$ such that $M_0=M_0(V)$.
According to Lemma \ref{th:4.6},
we can see that  $V\in \mathfrak{V}^{L}_{\mathfrak{n}}\cup\mathfrak{V}^{L}_{\mathfrak{i}}\cup\mathfrak{V}^{L}_{\mathfrak{ne}}$.
First we consider $V\in\mathfrak{V}^{L}_{\mathfrak{n}}$. We may assume that $V$ is a standard subspace of $L_{-1}$.
Suppose $V_1$ and $V_3$ are subspaces of $L_{-1}$
associated with $V$ which are defined as in \eqref{4.2},  and $\mathbf{I}_{v_1}$,
$\mathbf{I}_{v_3}$ are corresponding index sets.
In this case, $L_{-1}=V_1\oplus V_3$, where $V_1=V$.
It is easy to show that $M_i({L_{-1}, M_0(V)})=V^{i+2}\oplus V_3^{i+2}$ by induction.
Furthermore,
$$M({L_{-1}, M_0(V)})=\Bigg(\bigoplus_{i>0} V^{i}\Bigg)\oplus\Bigg(\bigoplus_{j>0} V_3^{j}\Bigg).$$
Suppose $R(V)=2k$.
We have
$$
\dim M({L_{-1}, M_{0}(V)})=(2p)^k+(2p)^{n-k}-2.
$$
For $1< r< n-1$,
set
$$
V_r=\mathrm{span_{\mathbb{F}}}\{x_j\in L_{-1} \mid j\in\overline{1,r}\cup\overline{1',r'}\}.
$$
According to Lemma \ref{th:1.3}, $M({L_{-1}, M_0(V_r)})\cong M({L_{-1}, M_0(V_{n-r}}))$.
Moreover,
we can see that
there are $
\lfloor(n-2)/2\rfloor
$ pairwise maximal R-subalgebras up to isomorphism.

Next we consider $V\in \mathfrak{V}^{L}_{\mathfrak{i}}\cup\mathfrak{V}^{L}_{\mathfrak{ne}}$.
Notice that if $V\in \mathfrak{V}^{L}_{\mathfrak{ne}}$, then there exists $V'\in  \mathfrak{V}^{L}_{\mathfrak{i}}$ such that
$M_0(V)=M_0(V')$. Thus
we  may assume that $V\in \mathfrak{V}^{L}_{\mathfrak{i}}$ is a standard subspace of $L_{-1}$.
Suppose $V_2,\overline{V}_2$ and $V_3$ are subspaces of $L_{-1}$
associated with $V$ which are defined as in \eqref{4.2},  and $\mathbf{I}_{v_2}, \mathbf{I}_{\overline{v}_2}$ and
$\mathbf{I}_{v_3}$ are corresponding index sets.
Obviously, $M_0(V)=VL_{-1}\oplus V_3^2$.
Furthermore, by induction we have
$$
M(L_{-1}, M_0(V))=\Bigg(\bigoplus_{i\geq 0,j\geq0}V^iV_3^j\Bigg)\oplus\Bigg(\bigoplus_{l\geq0}V^l\overline{V}_2\Bigg).
$$
Moreover
$$
\dim M(L_{-1}, M_0(V))=p^{n-s+t}2^{n-t}+p^{t}2^{s-t}s-1,
$$
where
$\mathrm{misdim}V=s$
and $\dim V_0=t$.
Computing  the dimension of $M_0(V)$, we can see that in this case,
$$
M(L_{-1}, M_0(V))\stackrel{\mathrm{alg}}{\cong} M(L_{-1}, M_0(\widetilde{V}) )\Longleftrightarrow
V\stackrel{\mathrm{sspace }}{\cong} \widetilde{V}.
$$
Thus, there are $(n^2+n)/2$ pairwise maximal R-subalgebras up to isomorphism.
Note that
$$M(SHO_{-1},M_0(V))=M(HO_{-1},M_0(V))\cap SHO,$$
and
$M(SHO_{-1},M_0(V))$ is maximal in $SHO$ if and only if $M(HO_{-1},M_0(V))$ is maximal in $HO$.
So to finish the theorem, we only need to compute the dimension of $M(SHO_{-1},M_0(V))$.
Suppose $\mathrm{R}(V)=2k$,
$\mathrm{misdim}V=s$
and $\dim V_0-k=t$.

If $V\in \mathfrak{V}_{\mathfrak{n}}^{SHO}$,
then
\begin{align}
&\dim M(SHO_{-1}, M_0(V))\nonumber\\
=&\dim M(HO_{-1}, M_0(V))\cap SHO\nonumber\\
=&\dim SHO'(k)+\dim SHO'(n-k)\nonumber\\
=&\sum^{k}_{l=2}\bigg((2^{k-1}-2^{k-l})\sum_{(i_1,\ldots,i_l)\in\mathbf{J}(l,k)}\prod_{c=1}^l\pi_{i_c}\bigg)+(p+1)^k\nonumber\\
&+\sum^{n-k}_{m=2}\bigg((2^{n-k-1}-2^{n-k-m})\sum_{(j_1,\ldots,j_m)\in\mathbf{J}(m,n-k)}\prod_{c=1}^m\pi_{j_c}\bigg)+(p+1)^{n-k}-2.\nonumber
\end{align}

If $V\in \mathfrak{V}_{\mathfrak{i}}$,
then
\begin{align}
&\dim M(SHO_{-1}, M_0(V))\nonumber\\
=&\dim M(HO_{-1}, M_0(V))\cap SHO\nonumber\\
=&\dim \bigg(\bigoplus_{i>0} V_3^i\bigg)\oplus\Bigg(\bigg(\bigoplus_{j>0,q\geq0}V^jV_3^q\bigg)\cap SHO\Bigg)+ \dim \Bigg(\bigg(\bigoplus_{l\geq0}V^l\overline{V}_2\bigg)\cap SHO\Bigg)\nonumber\\
=&p^t2^{n-s}\dim SHO'(n-s)-\dim\{x^{(\alpha)}x^u\in\oplus_{i>1} V_3^i\mid\mathbf{I}(\alpha, u)=\widetilde{\mathbf{I}}(\alpha, u)=\phi\}-1\nonumber\\
&+(t-1)p^{t-2}(p-1)^22^s+(s-1)p^t2^{s-2}+p^{t-1}(p-1)2^{s-1}\nonumber\\
=&p^t2^{n-s}\Bigg(\sum^{n-s}_{i=2}\bigg((2^{n-s-1}-2^{n-s-l})\sum_{(i_1,\ldots,i_l)\in\mathbf{J}(l,n-s)}\prod_{c=1}^l\pi_{i_c}\bigg)+(p+1)^{n-s}-1\Bigg)\nonumber\\
&-2^{n-s}+(t-1)p^{t-2}(p-1)^22^s+(s-1)p^t2^{s-2}+p^{t-1}(p-1)2^{s-1}-1. & \qedhere \nonumber
\end{align}

\end{proof}

\subsection{$KO(n)$ and $SKO(n,\lambda)$}
Let $L=KO(n)$ or $SKO(n,\lambda)$.
In this subsection,
we will determine the  maximal R-subalgebras of type (\MyRoman{3}) for $L$.
 As for $HO(n)$ and $SHO(n)$, we have the following lemma:
\begin{lemma}\label{th4.11}
Let $A_{0}$ be a reducible maximal subalgebra of $L_{0}$.
If $V\subset L_{-1}$ is a nontrivial invariant subspace of $A_{0}$ such that $A_0=M_0(V)$,
then $M(L_{-1},A_{0})$ is an MGS of $L$ if and only if $V\in\mathfrak{V}^{L}_{\mathfrak{i}}\cup\mathfrak{V}^{L}_{\mathfrak{ne}}$.
\end{lemma}
\begin{proof}
Let $L=KO(n)$. Since $V$ is a invariant subspace of $A_{0}$,
we can see that $A_0=M_0(V)$.
Then as $M_0^{KO}(V)=M_0^{HO}(V)\oplus \mathbb{F}z$, we can see that
$V\in\mathfrak{V}_{\mathfrak{n}}^L\cup\mathfrak{V}^{L}_{\mathfrak{i}}\cup\mathfrak{V}^{L}_{\mathfrak{ne}}$.
We may assume that $V$ is a standard subspace of $L_{-1}$.
Suppose $V_1,V_2,\overline{V}_2$ and $V_3$ are subspaces of $L_{-1}$
associated with $V$ which are defined as in \eqref{4.2},  and $\mathbf{I}_{v_1},\mathbf{I}_{v_2}, \mathbf{I}_{\overline{v}_2}$ and
$\mathbf{I}_{v_3}$ are corresponding index sets.

We consider the sufficiency firstly.
Suppose $V\in\mathfrak{V}^{L}_{\mathfrak{i}}\cup\mathfrak{V}^{L}_{\mathfrak{ne}}$.
Assert that
$$
M(L_{-1},A_{0})\cap L_{1,1}, M(L_{-1},A_{0})\cap (L_{1,0}\backslash L'_{10}) \neq \emptyset.
$$
Indeed if $V\in\mathfrak{V}^{L}_{\mathfrak{i}}\cup\mathfrak{V}^{L}_{\mathfrak{ne}}$,
then there exists $l \in \mathbf{I}_{v_2}$ and $ m\in \overline{1,2n}$
such that
$x_{l}z\in L_{1,1}\cap M(L_{-1},A_{0})$
and $x_{l}x_{m}x_{m'}\in L_{1,0}\cap M(L_{-1}, A_{0})$ respectively.
Now we show that  $M(L_{-1}, A_0)$ is an MGS of $L$.
For any $D\in L\backslash M(L_{-1},A_{0}),$
there exists
$
D'\in L_{0}\backslash M_{0}(L_{-1},A_{0})
$,
by the definition of $M(L_{-1},A_{0})$.
We have $\langle{A_{0},D'\rangle}= L_{0}$,
as $A_{0}$ is maximal in $L_0$.
So
$\langle{M(L_{-1},A_{0}),D'\rangle}=L$ by our assertion.

Next we consider the necessity.
Suppose $V\notin\mathfrak{V}^{L}_{\mathfrak{i}}\cup\mathfrak{V}^{L}_{\mathfrak{ne}}$.
Then $V\in \mathfrak{V}^{L}_{\mathfrak{n}}$.
Now we show that $M(L_{-1},A_{0})$ is not an MGS of $L$.
For any $i>0$, set
$$
M_{i}^{0}=V^{i+2}\oplus V_3^{i+2}.
$$
Assert that $M_{i}^{0}=M_i(L_{-1}, A_0)$, for any $i>0$.
Firstly, we show that $M_{1}^{0}=M_1(L_{-1}, A_0)$.
Obviously, $M_1^0\subset M_1(L_{-1}, A_0)$.
Conversely, suppose there is
$
u\in M_1(L_{-1}, A_0)\backslash \mathrm{Ann}(L_{-2}).
$
We may assume that $u=x_iz+f_1$, where $x_i\in V$ and $f_1\in \mathrm{Ann}(L_{-2})$.
For any $t \in \overline{1,2n}$,
$$
[x_t,u] = [x_t,x_iz+f_1] = -(-1)^{\tau(i)}x_tx_i +[x_{t},x_i]z +[x_t,f_1]\in A_0.
$$
Consequently,
\begin{equation}\label{5.4}
-(-1)^{\tau(i)}x_tx_i + [x_t,f_1] \in A_0
\end{equation}
Suppose  $f_1 = f_1' +f_1'' +f_1'''$,
where $f_1'\in V^2V_3$, $f_1''\in VV_3^2$ and
 $f_1'''\in V^3\oplus V_3^3$.
For any $x_t \in V$, from \eqref{5.4}
we can see that
$$
[x_t,f_1'] + [x_t,f_1'']+ [x_t,f_1'''] -(-1)^{\tau(i)} x_t x_i \in A_0.
$$
So
$[x_t,f_1'] \in A_0$.
Since $VV_3\nsubseteq A_0$
and $[x_t,f_1']\in VV_3 $,
we have $[x_t,f_1'] =0$.
As $x_t$ can be any element in $V$,
we have $f_1' =0$.
Consequently, $f_1=f_1''+f_1'''$.
For any $x_t\in V_3$, we have
$$
-(-1)^{\tau(i)}x_tx_i +[x_t,f_1''] =0.
$$
Take $x_t\in (V_3)_{\bar 0}$. Then
$$
f_1'' =(-1)^{\tau(i)} x_t x_{t'}x_i + \widetilde{f}_1'',
$$
where $\partial_{t'}(\widetilde{f}_1'')=0.$
Moreover,
$\partial_t(\widetilde{f_1''}) =-(-1)^{\tau(i)}2x_{t'}x_i$,
since $\partial_t (f_1'') = -(-1)^{\tau(i)}x_{t'}x_i$.
So
$$
\partial_t \partial_{t'}(\widetilde{f_1''})= (-1)^{\tau(i)}2x_i \neq 0,
$$
which is contradict to $\partial_t \partial_{t'}(\widetilde{f_1''}) =0$.
Thus
$
 M_1(L_{-1}, A_0)\subset M_1^0.
$
For $i>1$, it is easily to see that
$M_i(L_{-1}, A_0)=M_i^0$ by induction.
Take $(s,t)\in \mathbf{I}_{v_2}\times\mathbf{I}_{v_3}$, then
$$
x_{s}x_{t}x_{t'}\in M_{1}(V)\setminus M_{1}(L_{-1}, A_{0}).
$$
Thus,
$
M(L_{-1}, A_{0})\subsetneq L_{-2}+L_{-1}+L_{0}+L_{1,0}+\cdots.
$
So $M(L_{-1},A_{0})$ is not an MGS of $L$.

For $L=SKO$, the lemma can be shown in a similar way.
\end{proof}
Let $A_{0}$ be a reducible maximal subalgebra of $L_{0}$.
Then there exists $V\subset L_{-1}$
such that $A_0=M_0(V)$.
Moreover, if $V\in\mathfrak{V}_{\mathfrak{ne}}^L$,
then there exists an isotropic subspace $V'$ of $V$ such that
$A_0=M(V')$.
Thus we only need to consider the case when $V\in \mathfrak{V}_{\mathfrak{i}}^L$
in order to enquire the maximal R-subalgebras of type (\MyRoman{3}).

Suppose $L=KO(n)$ or $SKO(n,\lambda)$.
Let $A_0$ be a maximal R-subalgebra of $L_{0}$.
Suppose $V\in \mathfrak{V}_{\mathfrak{i}}^L$ such that $A_0=M(V)$ and $\dim V=s$.
We  may assume that $V$ is a standard subspace of $L_{-1}$.
Suppose $V_1,V_2,\overline{V}_2$ and $V_3$ are subspaces of $L_{-1}$
associated with $V$ which are defined as in \eqref{4.2},  and $\mathbf{I}_{v_1},\mathbf{I}_{v_2}, \mathbf{I}_{\overline{v}_2}$ and
$\mathbf{I}_{v_3}$ are corresponding index sets.
Let $y_0=\sum_{i\in\mathbf{I}_{\overline{v}_2}}(-1)^{\tau(i)}x_ix_{i'}$.
We have the following lemma:
\begin{lemma}\label{th:4.12}
Suppose $u\in L $ and $u\notin \mathrm{Ann}(L_{-2})$.
\item[(1)]For $L=KO(n)$,
\begin{enumerate}
\item[(a)] if $s=n$, then $u\in M{(L_{-1},A_{0})}$ if and only if
$u=fz+fy_0+g$,
where $f, \ g\in \left(\bigoplus_{i\geq0}V^i\overline{V}_2 \right)\oplus\left(\bigoplus_{j\geq0}V^j\right)$.
\item[(b)]if $s\neq n$, then $u\in M{(L_{-1},A_{0})}$ if and only if $u=fz+fy_0+g$,
where
$f\in \oplus_{i\geq 0,j\geq0}V^iV_3^j, \ g\in\left(\bigoplus_{i\geq0}V^i\overline{V}_2\right)\oplus\left(\bigoplus_{i\geq 0,j\geq0}V^iV_3^j\right)$.
\end{enumerate}
\item[(2)] For $L=SKO(n,\lambda)$,
\begin{enumerate}
\item[(a)]if $n\lambda-t+s=0 ~\mbox{in}~ \mathbb{F}$,
then $u\in M{(L_{-1},A_{0})}$ if and only if
$u=f(z+n\lambda x_1x_{1'})_+g$, where
 $f\in \bigoplus_{i\geq0}V^i,  g\in\bigg(\Big(\bigoplus_{i\geq0}V^i\overline{V}_2\Big)\oplus\Big(\bigoplus_{i\geq 0,j\geq0}V^iV_3^j\Big)\bigg)\cap L$.
\item[(b)]if $n\lambda-t+s\neq 0  ~\mbox{in}~ \mathbb{F}$, then $u=k(z+n\lambda x_1x_{1'})+g$, where
$k\in \mathbb{F}, g\in V_1^2\oplus V_2L_{-1}\oplus V_3^2$.
\end{enumerate}
\end{lemma}

\begin{proof}
Here we only show the case that $L=KO(n)$ with $s=n$.
Suppose  $L=KO(n)$ with $s=n$ below.
First we consider the sufficiency.
Denote $\Omega=M(L_{-1}, A_0)$.
Suppose $u=fz+fy_{0}+g$ satisfying the condition (1)(a).
Evidently, $g\in \Omega$.
Thus we only need to show that $fz+fy_0\in \Omega$ by induction on $|f|$.
For $|f|=1$, we may  assume that  $f=x_{i},\ i\in \overline{1,2n}$.
For any $k\in  \overline{1,2n}$, we have
$$
[x_{k}, x_iz+x_iy_0]=[x_{k},x_i]z-(-1)^{\tau(i)}x_kx_i+(-1)^{\tau(i)}[x_{k},y_{0}]x_i+[x_{k},x_i]y_{0}\in A_0.
$$
Thus $x_iz+x_iy_0\in \Omega, \ \mbox{for}\  i\in \overline{1,2n}$.
Suppose $|f|>1$. For any $k\in  \overline{1,2n}$,
we see from induction that
$$
[x_{k}, fz+fy_0]=[x_{k},f]z-(-1)^{|f|}x_kf+(-1)^{|f|}[x_{k},y_{0}]f+[x_{k}, f]y_{0}\in \Omega.
$$
So $fz+fy_0\in \Omega$,
and consequently $u\in \Omega$.
Next we consider the necessity. Suppose $u=fz+g\in \Omega$, $f,g\in\mathrm{Ann}(L_{-2})$.
We will show that $u$ satisfies the condition (1)(a).
Since  $u\in \Omega$,
we have
$$
f=-(-1)^{|f|}[1,u]\in \Omega\cap HO(n)=\Bigg(\bigoplus_{i\geq0}V^i\overline{V}_2 \Bigg)\oplus\Bigg(\bigoplus_{j\geq0}V^j\Bigg).
$$
Hence $fz+fy_0\in \Omega$. Now $u=fz+fy_0+g'$, where $g'\in \Omega\cap HO$.
So $u$ satisfies the condition (1)(a).
\end{proof}
\begin{proof}[\textit{Proof of Theorem \ref{0.003}(2)}] It is a direct consequence  of  Lemma \ref{th:4.12}.
\end{proof}

In order to complete the description of all the MGS of Lie superalgebra for odd Cartan type,
it remains for us to describe the maximal S-subalgebras.
Unfortunately, we have not obtained their explicit description similar to other types.
However, we know that
$$
HO_0\cong \widetilde{\frak{p}}(2n)\cong SKO_0; \ SHO_0\cong \frak{p}(2n)
;\ KO_0\cong \widetilde{\frak{p}}(2n)\oplus \mathbb{F}I_{2n\times2n},
$$
where
\begin{align}
 \widetilde{\frak{p}}(2n)&=\Big\{\begin{pmatrix}
A & B\\
 C&-A^{T}
 \end{pmatrix}\in \mathfrak{gl}(2n,2n)\mid B=B^{T},\ C=-C^{T}\Big\},\nonumber\\
\frak{p}(2n)&=\Big\{\begin{pmatrix}
A & B\\
 C&-A^{T}
\end{pmatrix}\in \mathfrak{gl}(2n,2n)\mid B=B^{T},\ C=-C^{T}, \mathrm{tr}A=0\Big\}.\nonumber
\end{align}
We only need to determine all the  maximal subalgebras of $\frak{p}(2n)$, $\widetilde{\frak{p}}(2n)$ and $\widetilde{\frak{p}}(2n)\oplus \mathbb{F}I$,
by Lemmas \ref{th4.1} and \ref{th:4.2}.
So the classification of maximal S-subalgebras of $L$ can be reduced to
the classification  for the maximal subalgebras of $\frak{p}(2n)$, $\widetilde{\frak{p}}(2n)$ and $\widetilde{\frak{p}}(2n)\oplus \mathbb{F}I$.

\addcontentsline{toc}{chapter}{Appendix}
\appendix
\renewcommand\thesection{Appendix}
\section{}

\renewcommand\thesection{\arabic{section}}
\renewcommand{\thelemma}{A.\arabic{lemma}}
\numberwithin{equation}{section}
Assume that  the underlying field $\mathbb{F}$ is
of characteristic $p\neq2$.

Recall the definition of $\mathbb{Z}_2$-graded vector space.
A $\mathbb{Z}_2$-graded vector space is a vector space $V$
which decomposes into a direct sum of the form
$V=V_{\bar 0}\oplus V_{\bar 1}$,
 where $V_{\bar 0}$ and $V_{\bar 1}$ are vector spaces.
 Denote  the superdimension of $V$
 by $\mathrm{sdim}V=(\dim V_{\bar 0}, \dim V_{\bar 1})$.
The elements of $V_{\bar 0}$ (resp. $V_{\bar 1}$ ) are called even (resp. odd).
An element of $V$ is called $\mathbb{Z}_2$-homogeneous (or simply homogeneous)
if it is an  even or odd element.
A basis of $V$ is called a homogeneous basis
if each basis element is homogeneous.
To say $\{\alpha_1,\ldots,\alpha_m\mid\beta_1,\ldots,\beta_n\}$ is an ordered homogeneous basis of $V$
means $\alpha_i\in V_{\bar 0}$, $\beta_j\in V_{\bar 1}$.
A subspace $U$ of $V$ is called $\mathbb{Z}_2$-graded
if $U=\oplus_{\theta\in \mathbb{Z}_2}(U\cap V_{\theta})$.

We are now in the position to describe a $\mathbb{Z}_2$-graded vector space with a non-degenerate bilinear form.
From linear algebra theory,
a vector space $W$ is called a symplectic space,
if there is a non-degenerate, skew-symmetric bilinear form $(-,-)$ on $W$,
denoted as $(W, (-,-))$.

\begin{definition}\label{A.1}
A $\mathbb{Z}_2$-graded vector space $V$ is called a symplectic superspace,
if there is an odd, non-degenerate, skew supersymmetric bilinear form $(-,-)$ on $V$,
denoted as $(V, (-,-))$  or simply, $V$.
\end{definition}
\begin{remark}
 A subspace  of a symplectic superspace in general is  not a symplectic superspace.
\end{remark}
From the linear algebra theory,
if $(W, (-,-))$ is a symplectic space,
then there exists a basis of $W$
$$
\{\alpha_1, \alpha_2, \ldots, \alpha_n, \alpha_{1'}, \alpha_{2'},\ldots,
\alpha_{n'}\}
$$
such that $(\alpha_i, \alpha_{i'})=-(\alpha_{i'}, \alpha_{i})=1$  for $i\in \{1,\ldots,n\}$;
$(\alpha_i, \alpha_{j})=0$, for $i\neq j'$.
We have a similar fact as in non-super case:
\begin{lemma}\label{th:0.19}

Let $U$ be a subspace of a symplectic superspace $V$.
Then there exists a homogeneous basis of $U$
$$
\{\alpha_1, \ldots,\alpha _k, \alpha_{k+1}, \ldots, \alpha_l \mid \alpha_{1'}, \ldots, \alpha_{k'},\alpha_{(k+1)'},\ldots,\alpha_{s'}\}
$$
such that for
$i\in \{1,\ldots, l\}$ and $j\in\{1',\ldots,s'\}$
\begin{equation}\tag{A.3}
(\alpha_i, \alpha_{j})=
\begin{cases}
1,&  i \in \{1,\ldots,k\}, j=i'\\
0,& \text{otherwise}
\end{cases}
\end{equation}
where $2k$ is the rank of $U$ and $\mathrm{sdim}U=(l,s)$.
In particular, the dimension of $V$ is even.
\end{lemma}
\begin{proof}
We use induction on $\dim U$.
If $\dim U=1$, then the lemma obviously holds
since $(-,-)$ is odd.
Now assume that  $\dim U>1$.
If $(-,-) =0$, the lemma is obviously true.
Otherwise, there are  vectors  $\beta_1 \in V_{\overline 0}$, $\beta_{1'} \in V_{\overline 1}$
such that  $(\beta_1, \beta_{1'}) \neq 0$.
We may assume that $(\beta_1, \beta_{1'})=1$.
Extend $\{\beta_1, \beta_{1'}\}$ to a homogeneous basis of $U$
$$\{\beta_1,\beta_2', \ldots, \beta_l'\mid \beta _{1'}, \beta_{2'}', \ldots, \beta_{s'}'\}.$$
For all $i\in \{2\ldots l\}$, $j\in\{ 2\ldots s\}$,
put
$$
\beta_i = \beta_{i}' - (\beta_i', \beta_{1'})\beta_1, \
\beta_{j'} = \beta_{j'}' - (\beta_{j'}', \beta_1) \beta_{1'}.
$$
We have $(\beta_i, \beta_{1'}) =0 = (\beta_{j'}, \beta_1)$,
and $\{\beta_1,\beta_2, \ldots, \beta_l \mid \beta _{1'}, \beta_{2'}, \ldots, \beta_{s'}\}$ is
a homogeneous basis of $U$.
Set
$$U =U_1+U_2$$
where $U_1=\mathrm{span}_{\mathbb{F}}\{\beta_1 \mid \beta_{1'}\},$
$U_2=\mathrm{span}_{\mathbb{F}}\{\beta_2, \ldots, \beta_l \mid \beta_{2'}, \ldots, \beta_{s'}\}$.
By induction there exists a homogeneous  basis of $U_2$
$$\{\alpha_2, \ldots, \alpha_{k_1}, \alpha_{k_1+1}, \ldots, \alpha_l \mid \alpha_{2'}, \ldots, \alpha_{k_1'}, \alpha_{(k_1+1)'}, \ldots, \alpha_{s'}\}$$
such that  for $i\in \{2,\ldots, l\}$ and $j\in\{2',\ldots,s'\}$
 $$
(\alpha_i, \alpha_{j})=
\begin{cases}
1,&  i \in\{2\ldots k_1\}, j=i',\\
0,& \text{otherwise},
\end{cases}
$$
where $2(k_1-1)$ is the rank of $U_2$.
Putting $\alpha_1=\beta_1, \alpha_{1'}=\beta_{1'}$, one sees that
$$\{\alpha_1, \ldots,\alpha _{k_1}, \alpha_{k_1+1}, \ldots, \alpha_l\mid \alpha_{1'}, \ldots, \alpha_{{k_1'}},\alpha_{(k_1+1)'},\ldots, \alpha_{s'}\}$$
is a desired basis and $k_1=\mathrm{R}(U)/2=k$.
\end{proof}
Suppose $U$ is a subspace of a symplectic superspace $V$.
If a homogeneous basis of $U$ satisfies the identity (A.3),
we call it a homogeneous symplectic basis of $U$.
\begin{corollary}\label{th:001}
Let $V$ be a symplectic superspace.
Then there exists a homogeneous basis
$$
\{\alpha_1, \alpha_2, \ldots \alpha_n\mid \alpha_{1'}, \alpha_{2'},\ldots
\alpha_{n'}\}
$$
such that for
$i\in \{1\ldots n\}$, $j\in \{1'\ldots n'\}
$
\begin{equation*}
(\alpha_i, \alpha_{j})=
\begin{cases}
1,&   j=i',\\
0,& \text{otherwise},
\end{cases}
\end{equation*}
where $\dim V=2n$.
\end{corollary}

If $U$ is a $\mathbb Z_2$-graded subspace of $V$,
put $U^{\bot}=\{x\in V\mid (x, U)=0\}$,
then
$U^{\bot}$ is a $\mathbb Z_2$-graded subspace of $U$, called a complementary space of $U$.
Moreover,
$
\operatorname{dim}U + \operatorname{dim}U^{\bot} = \operatorname{dim}V,
$
if $(-,-)$ is nondegenerate on $V$.
\begin{definition}\label{A6}
Let $V$ be  a symplectic superspace, and $V_1$, $V_2$ be subspaces of $V$.
We call $V_1$ is sym-isomorphic to $V_2$
if there is an isomorphism $f: V_1\longrightarrow V_2$ as supervector space such that
$(u,v)=(fu, fv)$, for all $u, \ v \in V$,
denoted as $V_1\stackrel{\mathrm{sspace }}{\cong} V_2$.  We call $f$ a symplectic isomorphism.
\end{definition}
Let $V$ be a symplectic superspace.
Put $\mathrm{ssdim}V=(k,s,m,n)$, where $2k$ is the rank of $V$, $s$
is the dimension of maximal isotropic subspace of $V$, $m=\dim V_{\bar0}$,
$n=\dim V_{\bar1}$.
\begin{proposition}
Let $V$ be a symplectic superspace,
and $V_1$, $V_2$ be subspaces of $V$. Then
$$
V_1\stackrel{\mathrm{sspace }}{\cong} V_2\Longleftrightarrow \mathrm{ssdim}V_1=\mathrm{ssdim}V_2
$$
\end{proposition}

\begin{lemma}\label{0.8}
Let $V$ be a symplectic superspace.
Then $f:V\longrightarrow V$ is a sym-isomorphism if and only if
$K^TJK=J$, where
$K$ is the matrix of $f$ for any chosen homogeneous basis of $V$, denoted as $\begin{pmatrix}
A_{n\times n} & 0\\
 0&D_{n\times n}
\end{pmatrix}$, $\dim V=2n$,
$J=\begin{pmatrix}
 0& I_n\\
 -I_n&0
\end{pmatrix}$, that is, $A^TD=I$.
\end{lemma}

\begin{definition}
A $\mathbb Z_2$-graded subspace $W$ of a symplectic superspace $V$
is called isotropic (resp. symplectic) subspace of $V$
if $W \subset  W^{\bot}$ (resp. $W \cap W^{\bot}=\{0\}$).
\end{definition}

\begin{lemma}\label{th:0.4}
Let $V$ be a symplectic superspace,
and $W$ be an isotropic (symplectic) subspace of $V$.
Then any homogeneous basis of $W$ can be extended to a
homogeneous symplectic basis of $V$.
\end{lemma}

\begin{proof}
First we show that the lemma holds for $W= W^{\bot}$. In fact,
let $\mathrm{sdim}W=(s, t)$  and
$$
\{\alpha_{i_{1}}, \ldots, \alpha_{i_{s}}\mid \alpha_{j_{1}}, \ldots, \alpha_{j_{t}}\}
$$ be a homogeneous basis of $W$.
Put
$$
W_r = \mathrm{span}_{\mathbb{F}}\Big\{\{\alpha_{i_{1}}, \ldots, \alpha_{i_{s}}\mid\alpha_{j_{1}}, \ldots, \alpha_{j_{t}}\}\backslash \{\alpha_{r}\}\Big\},
$$
for any $r\in\Gamma=\{i_{1},\ldots,  i_{s}, j_{1}, \ldots, j_{t} \}$.
Since $W_r \subsetneq W$, we have $W = W^{\bot} \subsetneq W_r^{\bot}$.
Thus there is $\beta_{r} \in W_r^{\bot}\setminus W$.
We can assume that $\beta_{r}$ is a homogeneous element of $V$
such that $(\beta_{r}, \alpha_{r})=1$, since $\beta_{r}\notin W$.
Obviously $(\alpha_{r_{1}}, \beta_{r})=0$,
for any $r_{1}\in \Gamma \backslash \{r\} $.
Hence,
there is a basis of $W_r$
$$
\{\alpha_{i_{1}}, \ldots, \alpha_{i_{s}},\beta_{j_{1}}, \ldots, \beta_{j_{t}}\mid
\beta_{i_{1}}, \ldots, \beta_{i_{s}},\alpha_{j_{1}}, \ldots, \alpha_{j_{t}}\}
$$
such that for $i,\ j\in\{i_1,\ldots,i_s, j_1,\ldots,j_t\}$,
$$(\alpha_i, \beta_j)=
\begin{cases}
1, &  i=j,\\
0, &  i\neq j.
\end{cases}$$
For  $k\in \{1,\ldots,s\}$, $m\in \{1,\ldots,t\}$,
set
$$
\alpha_{i_{k}'}=\beta_{i_{k}}-\sum^{t}_{l=1}(\beta_{j_{l}},\beta_{i_{k}})\alpha_{j_{l}},\
\alpha_{j_{m}'}=\beta_{j_{k}}-\sum^{s}_{l=1}(\beta_{i_{l}},\beta_{i_{m}})\alpha_{i_{l}}.
$$
A short calculation shows that
$$\{\alpha_{i_{1}}, \ldots, \alpha_{i_{s}},\alpha_{j_{1}}', \ldots, \alpha_{j_{t}}'
 \mid \alpha_{i_{1}}', \ldots, \alpha_{i_{s}}',\alpha_{j_{1}}, \ldots, \alpha_{j_{t}}\}
$$
is a homogeneous symplectic basis of $V$.
The reminder is obvious  in view of the fact above.
\end{proof}

\end{document}